\documentclass[10pt, leqno]{article}

\usepackage{amsmath,amssymb,amsthm, epsfig}
\usepackage[colorlinks=true,urlcolor=blue,
citecolor=red,linkcolor=blue,linktocpage,pdfpagelabels,
bookmarksnumbered,bookmarksopen]{hyperref}
\usepackage{color}
\usepackage{ulem}
\usepackage{dsfont}
\usepackage{mathrsfs}
\usepackage{mathtools}
\usepackage{stackrel}
\usepackage{esint,enumerate}

\usepackage[left=2.5cm,right=2.5cm,top=2.5cm,bottom=2.5cm]{geometry}

\newcommand{\intav}[1]{\mathchoice {\mathop{\vrule width 6pt height 3 pt depth  -2.5pt
\kern -8pt \intop}\nolimits_{\kern -6pt#1}} {\mathop{\vrule width
5pt height 3  pt depth -2.6pt \kern -6pt \intop}\nolimits_{#1}}
{\mathop{\vrule width 5pt height 3 pt depth -2.6pt \kern -6pt
\intop}\nolimits_{#1}} {\mathop{\vrule width 5pt height 3 pt depth
-2.6pt \kern -6pt \intop}\nolimits_{#1}}}

\newtheorem{proposition}{Proposition}[section]
\newtheorem{lemma}{Lemma}[section]
\newtheorem{theorem}{Theorem}[section]
\newtheorem{remark}{Remark}[section]
\newtheorem{corollary}{Corollary}[section]
\newtheorem{definition}{Definition}[section]
\newtheorem{statement}{Statement}[section]

\numberwithin{equation}{section}

\title{\bf Lipschitz regularity of almost-minimizers for vectorial Alt-Caffarelli functionals in Orlicz spaces}
\author{Pedro Fellype Pontes\footnote{\noindent Zhejiang Normal University. School of Mathematical Sciences, Jinhua 321004 - People’s Republic of China. \noindent \texttt{E-mail address: fellype.pontes@gmail.com}}, \,\,\,\,\, Jo\~{a}o Vitor da Silva\footnote{\noindent Universidade Estadual de Campinas. Departamento de Matem\'{a}tica. Campinas, SP-Brazil 13083-859. \noindent \texttt{E-mail address: jdasilva@unicamp.br}}, \\  $\&$  \\Minbo Yang\footnote{\noindent  Corresponding author. Zhejiang Normal University.
School of Mathematical Sciences,
Jinhua 321004 - People's Republic of China.
 \noindent \texttt{E-mail address: mbyang@zjnu.edu.cn}}}
\date{Jan 2026}

\begin{document}

\maketitle

\begin{abstract}
For a fixed constant $\lambda > 0$ and a bounded Lipschitz domain $\Omega \subset \mathbb{R}^n$ with $n \geq 2$, we establish that almost-minimizers (functions satisfying a sort of variational inequality) of the Alt-Caffarelli  type functional 
\[
\mathcal{J}_G({\bf v};\Omega) \coloneqq \int_\Omega  \left(\sum_{i=1}^mG\big(|\nabla v_i(x)|\big) + \lambda \chi_{\{|{\bf v}|>0\}}(x)\right) dx ,
\]
where ${\bf v} = (v_1, \dots, v_m)$ and $m \in \mathbb{N}$, exhibit optimal Lipschitz continuity on compact subsets of $\Omega$, where $G$ is an $\mathcal{N}$-function satisfying specific growth conditions. Furthermore, we obtain universal gradient estimates for non-negative almost-minimizers in the interior of non-coincidence sets. %{\color{blue}Furthermore, under the additional convexity assumption on $G$, we address the problem of boundary Lipschitz regularity for $v$ by adopting a fundamentally different analytical approach.} Notably, this method also provides an alternative proof of the optimal local Lipschitz regularity in the domain's interior. 
Our work extends the recent regularity results for weakly coupled vectorial almost-minimizers for the $p$-Laplacian addressed in \cite{BFS24}, and even the scalar case treated in \cite{daSSV}, \cite{DiPFFV24} and \cite{PelegTeix24}, thereby providing new insights and approaches applicable to a variety of non-linear one or two-phase free boundary problems with non-standard growth.
\end{abstract}

\vspace{0.2cm}
	

\section{Introduction}\label{Sec1}

Let $\Omega \in \mathbb{R}^n$ be a bounded Lipschitz domain, with $n \ge 2$, $m \in \mathbb{N}$, $\lambda>0$ a fixed constant, and $G: [0,+\infty) \to [0,+\infty)$ an $\mathcal{N}$-function, we will deal with almost-minimizers of the Alt-Caffarelli type functional 
\begin{equation}\label{DefFunctional}
    \mathcal{J}_G({\bf v};\Omega) \coloneqq \int_\Omega  \left(\sum_{i=1}^mG\big(|\nabla v_i(x)|\big) + \lambda \chi_{\{|{\bf v}|>0\}}(x)\right) dx, 
\end{equation}
over the class
    $$\mathcal{K} \coloneqq \Big\{ {\bf v} \in W^{1,G}(\Omega;\mathbb{R}^m) \; : \; {\bf v} = {\bf g} \; \mbox{on} \; \partial \Omega \; \mbox{and} \; v_i \ge 0 \Big\}. $$
In this case, we consider ${\bf v} = (v_1, \dots, v_m)$, $|{\bf v}| = \sqrt{(v_1)^2 + \cdots + (v_m)^2}$, and ${\bf g} = (g_1, \dots, g_m)$, $0 \le g_i \in W^{1,G}(\Omega)$. 

Furthermore, we would like to point out that we understand as (local) \underline{$(\kappa,\beta)$-almost-minimizer} for $\mathcal{J}_G$ in $\Omega$, with constant $\kappa\le\kappa_0$ and exponent $0<\beta<1$, a vectorial function ${\bf u} = (u_1, \dots, u_m)$ satisfying
\begin{equation}\label{VarIneqAlmMin}
     \mathcal{J}_G({\bf u};B_r(x_0)) \le \big(1+\kappa r^\beta\big)  \mathcal{J}_G({\bf v};B_r(x_0)),
\end{equation}
for any ball $B_r(x_0)$, with $0<r<1$, such that $\overline{B_r(x_0)} \subset \Omega$, and any ${\bf v} \in W^{1,G}(B_r(x_0);\mathbb{R}^m)$ such that ${\bf u} = {\bf v}$ on $\partial B_r(x_0)$. 

It is important to clarify the restriction to the case $0 < \beta < 1$ in the definition of $(\kappa,\beta)$-almost-minimizer. Since we are interested in establishing local regularity estimates, we consider balls $B_r(x_0) \subset \Omega$ with radii $0 < r < 1$. Within such a setting, any function that is a $(\kappa,\beta_1)$-almost-minimizer for some $\beta_1 \ge 1$ automatically satisfies the $(\kappa,\beta)$-almost-minimizing condition for all exponents $0 < \beta < 1$. Consequently, the class of $(\kappa,\beta)$-almost-minimizers with $0 < \beta < 1$ naturally includes all cases with higher exponents, provided we restrict our attention to sufficiently small exponents.

We emphasize that, from a geometric perspective, the function $\psi(r)=\kappa r^{\beta}$ quantifies the deviation from minimality of almost-minimizers of~\eqref{VarIneqAlmMin} relative to weak solutions of the Euler--Lagrange equation associated with a vectorial $g$-Laplacian operator (that is, variational minimizers). This viewpoint parallels classical ideas in the regularity theory of oriented boundaries in $\mathbb{R}^n$ that are almost minimal for the area functional, namely sets satisfying
\[
\mathrm{Per}(\partial E \cap A, B_r(x))
\le
\mathrm{Per}(\partial F \cap A, B_r(x))
+ \mathtt{``small\ error"}(n,r),
\qquad (F \Delta E \Subset A),
\]
as developed by Tamanini in the 1980s. We refer the interested reader to~\cite{T} for an enlightening survey on this topic.

\bigskip

To present the main results of this paper, we need two definitions. For more details, we refer the reader to Section \ref{Sec2}.

\begin{definition}[{\bf $\mathcal{N}$-function}]\label{Def-N-function}
    A continuous function $G : [0,+\infty) \rightarrow [0,+\infty)$ is an $\mathcal{N}$-function if:
	\begin{itemize}
		\item[$(i)$] $G$ is convex;
		\item[$(ii)$] $G(t) = 0$ if, and only if, $t = 0$;
		\item[$(iii)$] $\displaystyle\lim_{t\rightarrow0}\frac{G(t)}{t}=0$ and $\displaystyle\lim_{t\rightarrow+\infty}\frac{G(t)}{t}= +\infty$.
	\end{itemize}
\end{definition} 

\begin{definition}[{\bf Non-degenerate classes - \cite{BM}}]\label{defclasses} 
Let $G: [0,+\infty) \to [0,+\infty)$ be an $\mathcal{N}-$function such that
\begin{equation}\label{Lieberman-Cond01}
G^{\prime}(t)=g(t) \quad \text{for a function} \quad g \in C^0([0,+\infty))\cap C^1((0,+\infty)).
\end{equation}
and for $0 < \delta \leq  g_0$ fixed parameters,
\begin{equation}\label{Ga}
		0<\delta\le \dfrac{g^{\prime}(s)s}{g(s)} \le g_0.
\end{equation}
Furthermore, for any $t>0$ and $\varpi>0$ suppose the following control
    \begin{equation}\label{Qg}
        \int_{t}^{t+\varpi} |\mathcal{Q}^{\prime}_g(s)| \, ds  \le \dfrac{\mathrm{C}_0}{t^{\nu}} \cdot \varpi^{\nu},
    \end{equation}
where $\mathcal{Q}_g(s)=\frac{g^{\prime}(s)s}{g(s)}$ and $\mathrm{C}_0>0$ depends only on $\delta,g_0$, and $\nu$. In this context, we may define, for any $\varepsilon_0>0$, the following non-degenerate classes:

$$
\left\{\begin{array}{l}
     \mathcal{G}(\delta,g_0) \coloneqq \big\{
G :[0,\infty) \to [0,\infty)\; :\; G \ \text{is an $\mathcal{N}$-function satisfying \eqref{Lieberman-Cond01} and \eqref{Ga}}\big\};  \\
     \mathcal{G}_{\nu}(\delta,g_0) \coloneqq \big\{
G :[0,\infty) \to [0,\infty)\; :\; G\in \mathcal{G}(\delta,g_0)  \ \text{and satisfies \eqref{Qg}}\big\}; \\
    \mathcal{G}(\delta,g_0,\rho_0) \coloneqq \big\{
G \in \mathcal{G}(\delta,g_0)\; :\; G(1)\ge \rho_0\big\}; \\
    \mathcal{G}_{\nu}(\delta,g_0,\rho_0) \coloneqq \big\{
G \in \mathcal{G}_{\nu}(\delta,g_0)\; :\; G(1) \ge \rho_0\big\}.
\end{array}\right.
$$
\end{definition}

\subsection{Main regularity results}

Our first result establishes a universal H\"{o}lder continuous estimate (across the free boundary) for $(\kappa,\beta)$-almost-minimizers to \eqref{VarIneqAlmMin}.

\begin{theorem}[{\bf Local H\"{o}lder regularity of almost-minimizers}]\label{Holdereg}
    Assume that $G\in \mathcal{G}(\delta,g_0)$. Consider ${\bf u}=(u_1,\dots, u_m)$ a $(\kappa,\beta)$-almost-minimizer of $\mathcal{J}_G$ in $\Omega$, with some positive constant $\kappa \le \kappa_0$ and exponent $0<\beta<1$. Then, ${\bf u} \in C_{loc}^{0,\alpha}(\Omega; \mathbb{R}^m)$, for any $0<\alpha \le \frac{\delta}{g_0}$. More precisely, for any $\Omega^{\prime} \Subset \Omega$, there exists a constant $\mathrm{C}=\mathrm{C}(n,m,\beta,\kappa_0,\delta,g_0)>0$ such that
        $$
        \| {\bf u}\|_{C^{0,\alpha}(\Omega^{\prime};\mathbb{R}^m)} \le \mathrm{C}G^{-1}\left( \sum_{i=1}^m \xi_1\Big(|\nabla u_i|_{L_G(\Omega)}\Big)+\lambda\right).
        $$
\end{theorem}

\medskip

Now, considering the set
    $$P_{\bf u} \coloneqq \bigcup_{i=1}^m \{u_i>0\},$$
we understand the free boundary of an $(\kappa,\beta)$-almost-minimizer ${\bf u}$ by
    $$\mathfrak{F}(\mathbf{u}) \coloneqq \partial P_{\bf u} \cap \Omega.$$
We address a H\"{o}lder gradient estimate (far away from the free boundary). For this purpose, it is necessary to strengthen our assumptions on the ellipticity constants $\delta$ and $g_0$. The statement of the result is given below.

\begin{theorem}[{\bf $C^{1, \alpha}$ of almost-minimizers}]\label{gradHold}
     Suppose $G\in \mathcal{G}(\delta,g_0)$, and let ${\bf u} = (u_1, \dots, u_m)$ be a $(\kappa,\beta)$-almost-minimizer of $\mathcal{J}_G$ in $\Omega$, with some positive constant $\kappa \le \kappa_0$ and exponent $0<\beta<1$. Additionally, suppose that $g_0=\delta+\varepsilon$, for some given $\varepsilon = \varepsilon(n,m,\beta,\delta,g_0)\ge0$. Then ${\bf u}$ exhibits local $C^{1,\gamma}$-regularity in $P_{\bf u}$. More precisely, for any open set $\Omega^{\prime} \Subset P_{\bf u}$, there exists an exponent $\gamma = \gamma(\delta, g_0, n, \beta) > 0$ and a constant $\mathrm{C} = \mathrm{C}(\Omega^{\prime}, \delta, g_0, n, \kappa_0, \beta) > 0$ such that
        $$\|{\bf u}\|_{C^{1,\gamma}(\Omega^{\prime};\mathbb{R}^m)} \le \mathrm{C}\left( \sum_{i=1}^m\int_\Omega G\big(\vert\nabla u_i\vert\big) \ dx + \lambda\right).$$
\end{theorem}

However, it is worth noting in the previous theorem that if we require the exponent $\beta$ to be sufficiently large -- rather than $0 < \beta < 1$ -- then the condition $g_0 = \delta + \varepsilon$ can be removed from our assumptions. This relaxation highlights the interaction between the degree of regularity and the structural assumptions about the coefficients. For a detailed statement and comments of this refinement, we refer the reader to the Remark \ref{obs1} and Corollary \ref{corgradhold} below.

\medskip

As our last result, we obtain a universal Lipschitz estimate for $(\kappa,\beta)$-almost-minimizers of \eqref{VarIneqAlmMin}.

\begin{theorem}[{\bf Local Lipschitz regularity of almost-minimizers}]\label{MainThm03}  Suppose $G\in \mathcal{G}_\nu(\delta,g_0,\rho_0)$, with $g_0=\delta+\varepsilon$, for some given $\varepsilon = \varepsilon(n,m,\beta,\delta,g_0)\ge0$, and let ${\bf u} = (u_1, \dots, u_m)$ be a $(\kappa,\beta)$-almost-minimizer of $\mathcal{J}_G$ in $\Omega$, with some positive constant $\kappa \le \kappa_0$ and exponent $0<\beta<1$. Then, $u$ is locally Lipschitz continuous.
\end{theorem}

In a similar way to Theorem \ref{gradHold}, it is important to emphasize that the structural restriction $g_0 = \delta + \varepsilon$ can be withdrawn, provided the exponent $\beta$ is taken sufficiently large. This observation reinforces the delicate balance between the regularity of almost-minimizers and the growth conditions imposed on the function $G$. In particular, increasing the value of $\beta$--which quantifies the proximity of the functional to being an exact minimizer--has a compensatory effect that allows us to weaken the assumptions on the growth ratio $g_0/\delta$. This reflects a deeper phenomenon: the closer a function is to being a legitimate minimizer (in the variational sense), the less stringent the structural conditions on the integrand must be to ensure regularity.

\medskip

As a consequence of our results, we obtain local Lipschitz bounds for minimizers of $\mathcal{J}_G$. These estimates can be regarded as the local vectorial counterpart of those established by Mart\'{i}nez--Wolanski \cite{MW08} and da Silva \textit{et al.} \cite{daSSV} in the scalar setting.

\begin{corollary}[{\bf Lipschitz regularity for minimizers}]\label{Corollary1.1}
    Let $\Omega$ be a $C^{1,\alpha}$ domain in $\mathbb{R}^n$, and let $G$ be an $\mathcal{N}$-function such that $G \in \mathcal{G}(\delta,g_0)$ with $\delta>1$. Suppose $\mathbf{u}:\Omega \to \mathbb{R}^m$ is a minimizer of $\mathcal{J}_G$ in $\Omega$ with boundary data $\mathbf{\Phi}\in C^{1,\alpha}(\Omega,\mathbb{R}^m)$. Then $\mathbf{u}$ is locally Lipschitz continuous.
\end{corollary}

In particular, we emphasize that minimizers of $\mathcal{J}_G$, which are locally Lipschitz continuous by Corollary~\ref{Corollary1.1}, satisfy the associated one-phase Bernoulli-type free boundary problem
\[
\left\{
\begin{array}{rcl}
\mathrm{div}\!\left(g(|\nabla \mathbf{u}|)\dfrac{\nabla \mathbf{u}}{|\nabla \mathbf{u}|}\right) = 0 
& \text{in} & 
\displaystyle P_{\mathbf{u}} \coloneqq \Bigl(\bigcup_{i=1}^{m} \{u_i>0\}\Bigr) \cap \Omega,\\[0.2cm]
\mathbf{u} = 0, \quad |\nabla \mathbf{u}| = \lambda^{\sharp} 
& \text{on} & 
\partial P_{\mathbf{u}} \cap \Omega,\\[0.2cm]
G'(t) = g(t) 
& \text{and} & 
g(\lambda^{\sharp})\lambda^{\sharp} - G(\lambda^{\sharp}) = \lambda.
\end{array}
\right.
\]
This result extends the classical Alt--Caffarelli theory to the vectorial Orlicz--Sobolev framework (cf. Martinez--Wolanski \cite{MW08} for the scalar case). Moreover, it plays a central role in the analysis of further related one-phase free boundary problems, which we plan to address in forthcoming works (for the vectorial setting), including:
\begin{enumerate}
    \item[\checkmark] Optimization problems with volume constraints in Orlicz spaces (cf. \cite{AAC86} and \cite{Mart2008}),
    \[
    \min_{\mathcal{K}_{\mu}} \int_{\Omega} G(|\nabla u|)\,dx,
    \quad 
    \mathcal{K}_{\mu} \coloneqq \Bigl\{u \in W^{1,G}(\Omega) \,:\, |\{u>0\}|=\mu,\ \ u=\phi_0 \ \text{on } \partial\Omega \Bigr\}.
    \]
    \item[\checkmark] Singular perturbation problems for operators with Orlicz-type growth (cf. \cite{MW2009}),
    {\scriptsize
    \[
    \mathscr{L}_gu^{\varepsilon}
    = \mathrm{div}\!\left(g(|\nabla u^{\varepsilon}|)\dfrac{\nabla u^{\varepsilon}}{|\nabla u^{\varepsilon}|}\right)
    = \zeta_{\varepsilon}(u^{\varepsilon}), 
    \qquad u^{\varepsilon} \ge 0,
    \]
    }
    with $u^{\varepsilon} \to u_0$ as $\varepsilon \to 0$, where $u_0$ satisfies in the weak sense
    {\scriptsize
    \[
    \left\{
    \begin{array}{ccl}
    \mathrm{div}\!\left(g(|\nabla u_0|)\dfrac{\nabla u_0}{|\nabla u_0|}\right) = 0 
    & \text{in} & \{u_0>0\}\cap \Omega,\\[0.1cm]
    u_0=0,\quad |\nabla u_0| = \lambda^{\ast}(g, \mathrm{M}) 
    & \text{on} & \partial\{u_0>0\}\cap \Omega,\\[0.1cm]
    G'(t) = g(t) 
    & \text{and} & g(\lambda^{\ast})\lambda^{\ast}-G(\lambda^{\ast}) = \lambda
    \end{array}
    \right.
    \]
    }
    and
    \[
    \zeta_{\varepsilon}(s) = \frac{1}{\varepsilon}\,\zeta\!\left(\frac{s}{\varepsilon}\right),
    \qquad 
    0 \le \zeta \in C_0^{\infty}(0,1), 
    \quad 
    \int_0^1 \zeta(t)\,dt = \mathrm{M}>0,
    \]
    is an approximation of the identity.
\end{enumerate}

\subsubsection*{Structure of the manuscript}

The paper is organized as follows. In Subsection \ref{SubSec1}, we provide a brief overview of the current state-of-the-art related to the problem under consideration. Section \ref{Sec2} is dedicated to recalling essential definitions and fundamental properties of Orlicz and Orlicz-Sobolev spaces, which serve as the functional framework for our analysis. In Section \ref{Sec3}, we establish H\"{o}lder continuity results for almost-minimizers, namely Theorems \ref{Holdereg} and \ref{gradHold}. Finally, Section \ref{Sec4} is devoted to the proof of Theorem \ref{MainThm03}.

\subsubsection*{Novelties, comparisons, and technical challenges of this article
}

Although our approach is strongly inspired by \cite{BFS24}, we emphasize that the non-homogeneous nature of our functional introduces several additional difficulties, compounding the nonlinear features already present in the vectorial $p$-Laplacian setting. By way of illustration, we highlight the following issues:

\begin{enumerate}
    \item[\checkmark] The lack of sharp embedding results and the inherent difficulty of handling modulars and norms (for instance, in the application of H\"older's inequality);
    \item[\checkmark] The absence of established H\"older continuity estimates for vectorial almost-minimizers in Orlicz spaces (see Lemmas~\ref{lemreg} and~\ref{controlVg} for new results);
     \item[\checkmark] We obtain the precise control of the decay of the ``excess functional'', which yields H\"older gradient estimates for almost-minimizers away from the free boundary (see Lemma \ref{controlVg});
     \item[\checkmark] To derive several sharp estimates (e.g., Theorems~\ref{Holdereg} and~\ref{gradHold}), we employ a $g$-harmonic approximation scheme, which replaces the solution with a standard $g$-harmonic function and enables the application of known regularity results (see Lemma \ref{controlVg});
    \item[\checkmark] One novelty of our work is the weak compactness obtained via Caccioppoli-type inequalities (see Proposition~\ref{CaccioIne});
    \item[\checkmark] Another novelty, in the vectorial setting, is the convergence analysis of families $(\mathbf{v}^j, G_j)$ within a suitable topological framework (see Proposition~\ref{Prop4.2}), which is pivotal for establishing linear growth at free boundary points (see Proposition~\ref{MainProp02});
    \item[\checkmark] We obtain the uniform control of almost-minimizers near free boundary points (see Proposition~\ref{MainProp02});
    \item[\checkmark] Our approach differs substantially from the dichotomy-type strategy—namely, whether the averaged energy of an almost-minimizer decreases in smaller balls or the almost-minimizer is arbitrarily close to a linear function—employed by Dipierro \textit{et al.}~\cite{DiPFFV24} and da Silva \textit{et al.}~\cite{daSSV} to obtain Lipschitz regularity in scalar settings. Instead, we develop a collection of potential-theoretic tools (including Caccioppoli estimates, H\"older-type characterizations of almost-minimizers, $g$-harmonic replacements, and control of excess functionals), which provide the key mechanisms for establishing the desired Lipschitz regularity.
\end{enumerate}

These points summarize several of the main challenges addressed in this manuscript, in addition to the intrinsic technical difficulties arising from general non-power-type, and thus non-homogeneous, growth behaviors.

\subsection{State-of-the-Art: Alt-Caffarelli functionals and their regularity theories}\label{SubSec1}

Regarding minimizers in the linear setting, in Alt-Caffarelli's seminal paper \cite{AltCaf}, concerning the regularity of minimizers to the one-phase Bernoulli energy functional is given by
$$
\displaystyle \mathcal{J}_2(u, \Omega) := \int_{\Omega} \left(|\nabla u|^2 + \chi_{\{u>0\}}\right)dx,
$$
 where the local Lipschitz regularity estimate is established. Moreover, it can be shown that minimizers $\displaystyle \mathcal{J}_2(u_0, \Omega) = \min_{\mathcal{K}} \mathcal{J}_2(u, \Omega)$ are solutions of the following free boundary problem
\begin{equation}\label{AltCaffProb}
  \left\{
\begin{array}{rclcl}
  \Delta u_0(x) & =& 0 & \text{in} & \{u_0>0\}\cap \Omega \\
  |\nabla u_0| & = & \sqrt{2} & \text{on} & \partial \{u_0>0\}\cap\Omega\\
  u(x) & =& g(x) & \text{on} & \partial \Omega.
\end{array}
\right.
\end{equation}
in an appropriate distributional sense.

The natural motivations to investigate such a class of free boundary problems of Bernoulli type  \eqref{AltCaffProb} come from the analysis of cavity and jet type problems, see e.g. \cite[Section 1.1]{CafSalBook}. Another relevant model also arises in combustion theory \cite{BCN90}, optimal design problems \cite{Tei10}, optimization problems with constrained volume \cite{AAC86} and \cite{Mart2008}, shape optimization problems \cite{BucBut05} and \cite{BucVel15} and phase transitions \cite{PetVald05}, and \cite{PetVald05-II}, just to mention a few scenarios. In the geometric setting, the variational formulation of \eqref{AltCaffProb} enforces a balance between the Dirichlet energy and the measure term in the functional $\mathcal{J}_2$, thereby characterizing the free boundary $\Sigma_u \coloneqq \partial \{u>0\}$ as a hypersurface with prescribed mean curvature in a variational sense. Indeed, domain variations yield a free boundary condition in which the difference between the Dirichlet energy density and the measure term determines the mean curvature of $\Sigma_u$. Consequently, the regularity and structure of the free boundary are closely connected to classical problems in geometric measure theory and the calculus of variations, such as minimal surface and capillarity-type problems; see \cite{Velichkov23} for further discussion.

Over the past decades, there has been extensive literature exploring this research topic, particularly in the context of both one and two-phase problems, most notably the celebrated viscosity approach to the associated free boundary developed by Caffarelli in the series of seminal works \cite{Caff87}, \cite{Caff88} and \cite{Caff89}. We must refer the reader to \cite{DephiSpoVel21}, \cite{DeSFS19}, and \cite{Velichkov23} for comprehensive modern essays on the subject.

For a quasilinear scenario, minimizers of the functional 
$$
\mathcal{J}_p(u, \Omega) := \int_{\Omega} \left(|\nabla u|^p +\lambda^{p}\chi_{\{u>0\}}\right)dx, \quad 1<p< \infty \quad \Rightarrow \quad 
\left\{\begin{array}{ccl}
    \Delta_p u = 0 & \text{in}&  \{u>0\}\cap \Omega \\
    u=0, \,\,\,\,|\nabla u| = \frac{\lambda}{(p-1)^{\frac{1}{p}}} & \text{on} & \partial \{u>0\}\cap \Omega
\end{array}
\right.
$$ 
have also been considered by Danielli and Petrosyan in \cite{DanPet05}, in which the regularity of the free boundary near flat points was addressed. Recently, Lipschitz estimates of the minimizers of the functional $\mathcal{J}_p$ have been discussed by Dipierro and Karakhanyan in \cite{DiPKar18}, where the authors also supplied the proof of the Lipschitz regularity when $p = 2$ without relying on monotonicity formulas.

We must also recall Martinez-Wolanski's work \cite{MW08}, which considers the optimization problem of minimizing
$$
\mathcal{J}_{\mathrm{G}, \lambda}(u,\Omega) := \int_\Omega \left(\mathrm{G}(|\nabla u|)+\lambda\chi_{\{u>0\}}\right)\,dx  \quad \Rightarrow \quad 
\left\{\begin{array}{ccl}
    \mathrm{div} \left(g(|\nabla u|)\dfrac{\nabla u}{|\nabla u|}\right) = 0 & \text{in}&  \{u>0\}\cap \Omega \\
    u=0,\,\,\,\,\, |\nabla u| = \lambda^{\ast} & \text{on} & \partial \{u>0\}\cap \Omega\\
    G^{\prime}(t) = g(t)& \text{and} & g(\lambda^{\ast})\lambda^{\ast}-G(\lambda^{\ast}) = \lambda
\end{array}
\right.
$$
in the class $W^{1,\mathrm{G}}(\Omega)$ with $u-\phi_0\in W_0^{1,\mathrm{G}}(\Omega)$, for a bounded function $\phi_0\ge 0$ and $\lambda>0$. The authors prove that solutions to the optimization problem are locally Lipschitz continuous. Moreover, such solutions satisfy the corresponding free boundary problem of Bernoulli-type, thereby extending Alt-Caffarelli's results for the scenario of the Orlicz-Sobolev framework. Additionally, they address Caffarelli’s classification scheme: flat and Lipschitz free boundaries are locally $C^{1,\alpha}$ for some $\alpha(\verb"universal") \in (0,  1)$. In the sequel, in \cite{BM}, the authors extend the previous Mart\'{i}nez-Wolanski results to the two-phase problems and demonstrate the existence of a universal tolerance, depending only on the degenerate ellipticity and other intrinsic parameters, for the density of the negative phase along the free boundary, under which uniform Lipschitz regularity is attained, their approach strongly rely on the Karakhanyan's ideas in \cite{Karakhanyan2008}, which previous has addressed a similar result (local Lipschitz estimates) for a two-phase problem driven by $p-$Laplacian under a smallness condition on the density for the negativity set (see also \cite{Braga2018} for regularity estimates and asymptotic behavior of the free boundary for certain classes of minima of inhomogeneous two-phase Alt-Caffarelli functionals in Orlicz spaces). 
 %In \cite{BLO20}, the authors extend these results to quasilinear singular/degenerate operators in the nonhomogeneous setting (i.e., with a nonzero right-hand side).

Furthermore, viscosity approaches for one-phase problems driven by the $p-$Laplace and $p(x)-$Laplace operators have been developed in \cite{LeiRic18} and \cite{FerLed23}, respectively. Additionally, in \cite{DaSRRV23}, the authors obtain the Lipschitz regularity of viscosity solutions of one-phase problems with non-homogeneous degeneracy (fully nonlinear operators with double phase signature), and some regularity properties of their free boundaries were addressed.

\medskip

Returning to the setting of almost-minimizers, we note that these functions satisfy only a variational inequality (see~\eqref{VarIneqAlmMin}), rather than an exact Euler--Lagrange equation. Consequently, the associated profile $u$ (the almost-minimizer) minimizes the underlying geometric functional only up to a controlled error of order $\kappa r^{\beta}$. In the terminology of geometric measure theory, this means that $u$ can be interpreted as a quasi-minimal (or almost-minimal) ``hypersurface'', with deviation from minimality in $B_r(x_0)$ of order $\kappa r^{\beta}$. 
%We refer the reader to Tamanini's survey~\cite{T} for the %corresponding theory of almost-minimal boundaries, namely sets $E %\subset \Omega$ satisfying
%\[
%\mathrm{Per}(\partial E \cap A, B_r(x))
%\le
%\mathrm{Per}(\partial F \cap A, B_r(x)) + \omega(r)\, r^{n-1},
%\]
%for every $A \Subset \Omega$, $x \in A$, and every competitor $F$ %such that $F \Delta E \Subset A$, where $\omega(r)=o(1)$ and $X %\mapsto \mathrm{Per}(\partial X, B_r(x))$ denotes the perimeter %functional relative to the ball $B_r(x)$.
%{\color{red} Now coming back to the setting of almost-minimizers, we note that these functions satisfy only a variational inequality, see \eqref{VarIneqAlmMin}, rather than an exact Euler–Lagrange equation. As a consequence, the associated free boundary (\textcolor{blue}{free boundary or graph of $u$?})$\Sigma_u$ ``minimizes'' the underlying geometric functional only up to a controlled error of order $r^\beta$. In the terminology of geometric measure theory, this means that $\Sigma_u$ is a quasi-minimal (or almost-minimal) hypersurface at scale $r$. We refer the reader to the Tamanini's survey \cite{T} for the corresponding results on almost minimal boundaries}.
Thus, one of the main obstacles to establishing regularity properties for~\eqref{VarIneqAlmMin} is the absence of a monotonicity formula, as is available for minimizers (cf.~\cite{AltCaf}). In this regard, it seems to be challenging to derive such techniques to the setting of almost-minimizers (cf. \cite{DET19} and \cite{DEST21} for related topics). In particular, in the quite recent article \cite{DeSS21}, De Silva and Savin developed a non-variational approach, based on Harnack-type inequality, for profiles that do not necessarily satisfy an infinitesimal equation.

We highlight that almost-minimizers of $\mathcal{J}_2$ were widely investigated recently by several authors. In effect, we must cite the David {\it et al.} work \cite{DET19}, whereby developing an original approach combining techniques from potential theory and geometric measure theory, the authors obtain uniform rectifiability of the free boundary and, in the one-phase scenario, the corresponding $C^{1, \alpha}$ almost everywhere regularity. Therefore, Alt-Caffarelli's classical results in \cite{AltCaf} were extended to the framework of almost-minimizers.  We refer the reader to \cite{DEST21} for generalizations concerning variable coefficients. %Moreover, the analysis of the semilinear scenario with variable coefficients
%$$
%\displaystyle \mathcal{J}_{\mathbb{A}}^{\gamma}(v; \Omega) = \int_{\Omega} (\langle \mathbb{A}(x)\nabla v, \nabla v\rangle + q_{+}(v^{+})^{\gamma} + q_{-}(v^{-})^{\gamma})dx, \quad \text{for} \quad 0\le \gamma \le 1 \quad \text{and} \quad q_{\pm} \ge 0,
%$$
%where $\mathbb{A}$ is a matrix with H\"{o}lder continuous coefficients was explored in \cite{QueiTav18}, thereby proving sharp gradient estimates.

Recently, Pelegrino-Teixeira in \cite{PelegTeix24} investigated regularity estimates of quasi-minima of the Alt-Caffarelli energy functional
$$
\displaystyle \mathcal{J}_2(u, \Omega) := \int_{\Omega} \left(|\nabla u|^2 + \chi_{\{u>0\}}\right)dx
$$
Given a functional $\mathcal{J}: \Xi(\Omega) \to \mathbb{R}$ (for a suitable functional space $\Xi(\Omega)$), and a real parameter $\mathrm{Q} \geq 1$, we say $u \in \Xi(\Omega)$ is a \emph{quasi-minimizer} (Q-minimum for short) of $\mathcal{J}$ if
$$
\mathcal{J}(u, \mathrm{K}) \leq  \mathrm{Q}\cdot \mathcal{J}(v, \mathrm{K}),
$$
for all $v \in \Xi(\Omega)$ such that $\mathrm{K} := \text{supp}(u-v)$ is a compact subset of $\Omega$.
In such a context, they proved universal H\"{o}lder continuity of quasi-minima (including sign-changing Q-minima, see \cite[Theorem 3]{PelegTeix24}) and optimal Lipschitz regularity (for nonnegative Q-minimum) along their free boundaries (see \cite[Theorem 4]{PelegTeix24}).

In \cite{DeSS20}, using non-variational techniques, De Silva and Savin provided a different approach from that of \cite{DET19} and \cite{DT15} to deal with almost-minimizers of $\mathcal{J}_2$ and their free boundaries. Precisely, based on ideas developed by them in \cite{DeSS21}, they showed that almost-minimizers of $\mathcal{J}_2$ are viscosity solutions in a more general sense. Once this was confirmed, the regularity of the free boundary for almost-minimizers follows by applying De Silva's techniques, first developed in \cite{DeS11}.

Recently, Dipierro {\it et al.} in \cite{DiPFFV24} obtained a Lipschitz continuity result to nonnegative almost-minimizers for the nonlinear framework, namely for $\mathcal{J}_p$ in $B_1$, in the case $\max\left\{ 1, \frac{2n}{n+2}\right\}< p < \infty$. Precisely, there exists a universal constant $\mathrm{C}=\mathrm{C}(p, n, \beta, \kappa) >0$ such that
$$
\|\nabla u\|_{L^\infty(B_{1/2})}\leq \mathrm{C} \left(1+\|u\|_{W^{1,p}(B_1)}\right)
$$
Additionally, $u$ is uniformly Lipschitz continuous in a neighborhood of the contact set $\{u = 0\}$. In effect, the author's approach was strongly inspired by the method introduced by De Silva and Savin in \cite{DeSS20}. In their setting, the main obstacles faced are concerned with the lack of linearity and the loss of exact descriptions of $p-$harmonic profiles in terms of mean value properties. To overcome such complications, the authors exploit some regularity estimates available in the classical literature.
\medskip

Subsequently, Da Silva {\it et al.} in the manuscript \cite{daSSV} established the optimal Lipschitz regularity for non-negative almost-minimizers of a one-phase Bernoulli-type functional exhibiting non-standard growth. The functional under consideration is the scalar version of $\mathcal{J}_G$ presented in \eqref{DefFunctional} where $\Omega \subset \mathbb{R}^n$ is a bounded domain and $G$ is an $\mathcal{N}$-function satisfying appropriate structural conditions. The central result asserts that such almost-minimizers exhibit optimal local Lipschitz continuity.

To address the difficulties arising from the non-standard growth conditions, the authors employ Campanato-type estimates within the Orlicz-Sobolev framework, which represents a novel approach for this class of functionals. The principal outcome can be stated as follows: for any $\Omega^{\prime} \Subset \Omega$,
\[
\|\nabla u\|_{L^{\infty}(\Omega^{\prime})} \leq \mathrm{C} \, G^{-1}\left(1 + \int_{\Omega} G(|\nabla u|) \, dx\right),
\]
where $\mathrm{C}>0$ is a constant depending only on universal parameters. This work extends classical regularity theory for the one-phase Bernoulli problem to a more general class of variational functionals with non-polynomial growth (cf. \cite{DiPFFV24}), thereby offering new insights and methodologies relevant to a broad spectrum of nonlinear problems exhibiting non-standard growth behavior.

\bigskip

For the vectorial scenarios of the one-phase problem, we must quote the following contributions:

\begin{enumerate}

%  \item Recently, Fotouhi \textit{et al.} in \cite{Fotouhi2021} consider the elliptic system
%\[
%\Delta \mathbf{u} = \lambda^{+}(x) |\mathbf{u}^{+}|^{q-1} \mathbf{u}^{+} - \lambda^{-}(x) |\mathbf{u}^{-}|^{q-1} \mathbf{u}^{-} \quad \text{in} \quad B_1,
%\]
%where $\lambda^{\pm} > 0$ are  H\"{o}lder continuous functions, and $\mathbf{u} : B_1 \subset \mathbb{R}^n \to \mathbb{R}^m$ with $n \geq 2$ and $m \geq 1$. Observe that $\mathbf{u}^{\pm} = (u_1^{\pm}, \dots, u_m^{\pm})$ denotes the vectors of positive and negative parts of each component. 

%Weak solutions correspond to the unique minimizers (up to the prescribed boundary data) of the energy functional
%\[
%\mathcal{J}_0(\mathbf{u}) = \int_{B_1} \left( |\nabla \mathbf{u}|^2 + \frac{2}{1 + q} \lambda^{+}(x) |\mathbf{u}^{+}|^{1+q} + \frac{2}{1 + q} \lambda^{-}(x) |\mathbf{u}^{-}|^{1+q} \right) dx.
%\]

%The behavior of solutions, as well as that of the free boundary near asymptotically flat points, is analyzed following the framework developed by Andersson \textit{et al.} in \cite{ASUW15}. In addition, the authors establish higher regularity results for the solutions. Finally, they formulate an epiperimetric inequality, which serves as a principal tool for investigating the regularity of the free boundary $\partial\{x : |\mathbf{u}(x)| > 0\}$. As a consequence, the free boundary is shown to possess $C^{1, \alpha}$ regularity at asymptotically flat points.

  \item For a system with free boundary, De Silva {\it et al.} in \cite{DeSJeonShah22} study the regularity properties of vector-valued almost-minimizers of the functional
$$
      \displaystyle \mathcal{J}_2({\textbf{u}};\mathrm{D}) = \int_{\mathrm{D}} \left(|\nabla \textbf{u}(x)|^2 + 2|\textbf{u}|\right)dx, %\quad (\text{for}\,\,\,1 < p < \infty),
$$ %\textcolor{blue}{quien es $p$}
which is strongly related to a version of the classical obstacle problem (see \cite{Figalli18}).

% \item Afterwards, De Silva {\it et al.} in \cite{DeSJeonShah23} investigated the properties of almost minimizers for the functional
%\[
%W^{1,2}(\mathrm{D}; \mathbb{R}^m) \ni \mathbf{v} \mapsto \mathcal{J}^{\mathrm{F}}_2(\mathbf{v}, \mathrm{D}) \coloneqq \int_{\mathrm{D}} \left( |\nabla \mathbf{v}|^2 + 2\mathrm{F}(x, \mathbf{v}) \right) \, dx,
%\]
%where $\mathrm{D}$ is a domain in $\mathbb{R}^n$ and 
%\[
%\mathrm{F}(x, \mathbf{v}) = \frac{1}{1 + q} \left( \lambda^+(x) |\mathbf{v}^+|^{q+1} + \lambda^-(x) |\mathbf{v}^-|^{q+1} \right),
%\]
%for some exponent $q \in (0, 1)$, where the coefficients $\lambda^{\pm}$ are  H\"{o}lder continuous and uniformly bounded such that $0 < \lambda_0 \leq \lambda^{\pm}(x) \leq \lambda_1 < \infty$. %Here, $\mathbf{v}^{\pm} = (v^{\pm}_1, \dots, v^{\pm}_m)$ denotes the vectors composed of the respective positive and negative parts of each coordinate.

%Minimizers of this functional satisfy an elliptic system with free boundaries,
%\[
%\Delta \mathbf{u} = \lambda^{+}(x) |u^{+}|^{q-1} u^{+} - \lambda^{-}(x) |u^{-}|^{q-1} u^{-} \quad \text{in} \quad B_1,
%\]
%which has been previously analyzed in \cite{Fotouhi2021}. 

%The authors establish regularity estimates for both almost minimizers and their associated free boundaries. In particular, \cite[Theorem 1]{DeSJeonShah23}  demonstrates the optimal interior regularity of almost minimizers, namely $C_{\text{loc}}^{1, \alpha/2}$ regularity, along with corresponding quantitative estimates. 

  \item For the weakly coupled vectorial $p-$Laplacian, given constant $\lambda > 0$ and a bounded Lipschitz domain $\mathrm{D} \subset  \mathbb{R}^n$ (for $n \ge  2$), Shahgholian {\it et al.} in \cite[Theorem 1.1]{BFS24} address local optimal Lipschitz estimates for almost-minimizers of
$$
\displaystyle \mathcal{J}_p({\textbf{v}};\mathrm{D}) = \int_{\mathrm{D}} \left( \sum_{i=1}^{m} |\nabla v_i(x)|^p + \lambda \chi_{\{|\bf{v}|>0\}}(x)\right)dx, \quad (\text{for}\,\,\,1 < p < \infty),
$$
where ${\textbf{v}} = (v_1, \cdots , v_m)$, and $m \in \mathbb{N}$.

\item Finally, De Silva \textit{et al.} in \cite{DeSJeonShah25} investigated superlinear systems that give rise to free boundary problems. Such systems arise, for instance, from the minimization of the energy functional
\[
\mathcal{J}^{(p)}_2(\mathbf{u}, \Omega) \coloneqq \int_{\Omega} \left( |\nabla \mathbf{u}|^2 + \frac{2}{p} |\mathbf{u}|^p \right) \, dx, \quad 0 < p < 1,
\]
although solutions can also be interpreted in the viscosity framework.

First, they establish the optimal regularity of minimizers by employing a variational approach. Subsequently, they apply a linearization technique to derive the $C^{1,\alpha}$ regularity of the ``flat'' portion of the free boundary through a viscosity method. Finally, for free boundaries arising from minimizers, they extend this regularity result to analyticity.

\end{enumerate}

Summarizing, we should emphasize that our estimates are closely related to a
series of seminal works, both in the scalar and vectorial scenarios, as explained in the table below:

\begin{table}[h]
		\centering
		\resizebox{\textwidth}{!}{
			\begin{tabular}{c|c|c|c}
				{\bf {\tiny{Alt-Caffarelli functional}}} & {\bf {\tiny{{Profile}}}}  & {\bf {\tiny{Optimal regularity}}} &  \textbf{{\tiny{References}}} \\
				\hline
				{\tiny{$\displaystyle\int_{\Omega} \left(|\nabla u|^2 + \chi_{\{u>0\}}\right)dx$}} & {\tiny{Minimizer}} &  {\tiny{$C_{\text{loc}}^{0, 1}$}} &  {\tiny{\cite{AltCaf}}} \\
				\hline
				{\tiny{$\displaystyle \int_{\Omega} \left(|\nabla u|^p +\lambda^{p}\chi_{\{u>0\}}\right)dx$}} & {\tiny{Minimizer}} &  {\tiny{$C_{\text{loc}}^{0, 1}$}} &  {\tiny{\cite{DanPet05}}} \\
				\hline
				{\tiny{$\displaystyle \int_\Omega \left(\mathrm{G}(|\nabla u|)+\lambda\chi_{\{u>0\}}\right)\,dx$}} & {\tiny{Minimizer}} & {\tiny{$C_{\text{loc}}^{0, 1}$}}  &  {\tiny{\cite{MW08}}}   \\
				\hline
                {\tiny{$\displaystyle \int_{\Omega} \left(|\nabla u|^2 + \chi_{\{u>0\}}\right)dx$}} & {\tiny{Almost-minimizer}}&  {\tiny{$C_{\text{loc}}^{0, 1}$}} & {\tiny{\cite{DeSS20}}} \\
				\hline
				{\tiny{$\displaystyle \int_{\Omega} \left(|\nabla u|^2 + \chi_{\{u>0\}}\right)dx$}} & {\tiny{Quasi-minimizer}}&  {\tiny{$C_{\text{loc}}^{0, 1}$ along $\mathfrak{F}(u)$}} & {\tiny{\cite{PelegTeix24}}} \\
				\hline
				{\tiny{$\displaystyle \int_{\Omega} \left(|\nabla u|^p +\lambda^{p}\chi_{\{u>0\}}\right)dx$}} & {\tiny{Almost-minimizer}}& {\tiny{$C_{\text{loc}}^{0, 1}$}} & {\tiny{\cite{DiPFFV24}}}\\
				\hline
				{\tiny{$\displaystyle \int_{\Omega} \left(G(|\nabla u|) + \chi_{\{u>0\}}\right) \, dx$}} & {\tiny{Almost-minimizer}} & {\tiny{$C_{\text{loc}}^{0, 1}$}} & {\tiny{\cite{daSSV}}}\\
				\hline
				{\tiny{$\displaystyle  \int_{\Omega} \sum_{i=1}^{m} |\nabla v_i(x)|^p + \lambda \chi_{\{|\bf{v}|>0\}}dx$}} & {\tiny{Almost-minimizer}} & {\tiny{$C_{\text{loc}}^{0, 1}$}} & {\tiny{\cite{BFS24}}} \\
		\end{tabular}}
	\end{table}
%\newpage 
Taking into account the above results, we will derive uniform Lipschitz estimates for a general class of almost-minimizers of Alt-Caffarelli functionals in Orlicz spaces.

%Even if our approach is strongly inspired by the work \cite{BFS24}, we highlight that the non-homogeneous nature of our functional entails several difficulties that add to the nonlinear character of the problem already present in the vectorial $p-$Laplacian setting. By way of illustration, the lack of sharp embedding results and the hard task of handling modulars and norms (for instance, when applying H\"older's inequality), the absence of established H\"{o}lder continuity estimates for vectorial almost-minimizers in Orlicz spaces (see Lemmas \ref{lemreg} and \ref{controlVg}), the weak compactness derived via the Caccioppoli inequality (see Proposition \ref{CaccioIne}), the convergence analysis for a family of $(\mathbf{v}^j, G_j)$ within an appropriate topological framework (see Proposition \ref{Prop4.2}), and the uniform control of almost-minimizers near free boundary points (see Proposition \ref{MainProp02}) represent several of the significant challenges addressed in this manuscript. These are tackled alongside the inherent technical complexity of handling general non-power-type behaviors, which are not necessarily homogeneous.

\section{On the Orlicz and Orlicz-Sovolev spaces}\label{Sec2}

Here we recall some properties of the Orlicz spaces, which can be found in \cite{A,RR}. We recall that the definition of $\mathcal{N}$-function is already stated in the Introduction (see Definition \ref{Def-N-function}). Furthermore, by the convexity of any $\mathcal{N}$-function $G$, for any $t \in(0,+\infty)$, we get
	\begin{equation}\label{ine}
		G(\alpha t) \le \alpha G(t) \ \mbox{for} \ \alpha \in [0,1] \quad \mbox{and} \quad G(\tilde{\alpha} t) > \tilde{\alpha} G(t) \ \mbox{for} \ \tilde{\alpha}>1.
	\end{equation} 
    
\begin{definition}[{\bf Orlicz space}] The Orlicz space associated to $\mathcal{N}$-function $G$ and an open $U\subset \mathbb{R}^n$ is defined by
	$$
	L_G(U;\mathbb{R}^m)= \left\{{\bf u} \in L^1_{loc}(U;\mathbb{R}^m) \;: \;
	\int_U G \left(\frac{\vert {\bf u}(x)\vert}{\gamma}\right)\, dx<+\infty, \ \mbox{for some}~ \gamma>0\right\},
	$$
and
	$$
	|{\bf u}|_{L_G(U;\mathbb{R}^m)}=\inf \left\{\gamma>0\;:\;\int_U
	G \left(\frac{\vert {\bf u}(x)\vert}{\gamma} \right)\, dx\leq 1 \right\},
	$$
defines a norm (the Luxemburg norm) on $L_G(U;\mathbb{R}^m)$ and turns this space into a Banach space. In the study of the Orlicz space $L_G(U;\mathbb{R}^m)$, we denote by $K_G(U;\mathbb{R}^m)$ the Orlicz class as the set below
	$$K_G(U;\mathbb{R}^m) \coloneqq \left\{{\bf u} \in L^1_{loc}(U;\mathbb{R}^m)\;: \; \int_U G \big({\vert {\bf u}(x)\vert}\big)\, dx<+\infty\right\}.$$
In this case, we have
	$$K_G(U;\mathbb{R}^m) \subset L_G(U;\mathbb{R}^m),$$
where $L_G(U;\mathbb{R}^m)$ is the smallest subspace containing $K_G(U;\mathbb{R}^m)$.
\end{definition}

\begin{definition}[{\bf $\Delta_2$ and $\nabla_2-$condition}]
An $\mathcal{N}$-function $G$ satisfies the  $\Delta_2$-condition (shortly, $G\in \Delta_2$), if there exist $k_G>0$ and $t_0\geq0$  such that
	\begin{equation}\label{defD2}
		G(2t)\leq  k_GG(t),~  t\geq t_0.
	\end{equation}
Similarly, we say that an $\mathcal{N}$-function $G$ satisfies the  $\nabla_2$-condition (shortly, $G\in \nabla_2$), if there exist $l_G>1$ and $t_0\geq0$  such that
	\begin{equation}\label{defn2}
		G(t)\leq  \dfrac{1}{2l_G}G(l_Gt),~  t\geq t_0.
	\end{equation}
\end{definition}
It is possible to prove that if $G \in \Delta_2$ then
	$$K_G(U;\mathbb{R}^m) = L_G(U;\mathbb{R}^m).$$

The corresponding Orlicz-Sobolev space is defined by
    $$W^{1, G}(U;\mathbb{R}^m) = \left\{ {\bf u} \in L_{G}(U;\mathbb{R}^m) \ :\ \frac{\partial u_j}{\partial x_{i}} \in L_{G}(U;\mathbb{R}^m), \quad \substack{i = 1, \dots, n\\ j=1,\dots, m}\right\},$$
endowed with the norm
    $$\Vert {\bf u} \Vert_{W^{1,G}(U;\mathbb{R}^m)} =   \vert\nabla {\bf u}\vert _{L_G(U;\mathbb{R}^{nm})} +  \vert {\bf u}\vert_{L^G(U;\mathbb{R}^m)}.$$
The space $W^{1,G}(U;\mathbb{R}^m)$ endowed with $\|\cdot\|_{1,G}$ is always a Banach space. Furthermore, this space is reflexive and separable if, and only if, $G\in\Delta_2\cap\nabla_2$. For simplicity, when $m=1$ we will denote $L_G(U;\mathbb{R})$ and $W^{1,G}(U;\mathbb{R})$ solely as $L_G(U)$ and $W^{1,G}(U)$, respectively. 

Additionally, we have Poincar\'{e}'s inequality (see for instance \cite[Corollary 7.4.1]{HH}), namely
    \begin{equation}\label{Poincare}
        \intav{U} G\left(\tau\dfrac{|v-(v)_U|}{\mbox{diam} ( U)}\right)\ dx \le \intav{U} G(|\nabla v|)\ dx,
    \end{equation}
for any  $v \in W^{1,G}(U)$,  and some $\tau \in (0,1]$, where $\displaystyle (v)_U \coloneqq \intav{U} v(x) \ dx$.

We now turn our attention to a particular class of $\mathcal{N}$-functions, which brings to light natural conditions introduced by Lieberman (see \cite{L} for more details) in studying regularity
estimates of degenerate/singular elliptic PDEs of the type $-\mathrm{div} \left(g(|\nabla u|)\dfrac{\nabla u}{|\nabla u|}\right) = \mathcal{B}(x,u,\nabla u)$. This class, denoted by $\mathcal{G}(\delta, g_0)$ and defined in Definition \ref{defclasses}, is commonly referred to as the \textit{Lieberman class}. It is known to exhibit several desirable analytical properties. In particular, Lieberman \cite{L}, Martínez and Wolanski \cite{MW08}, and Fukagai, Ito, and Narukawa \cite{FIN} have established several important results concerning this class, where we summarize in the following statement:

\begin{statement}[{\bf Properties in the class $\mathcal{G}(\delta,g_0)$}]\label{Statement} Suppose $G\in \mathcal{G}(\delta,g_0)$. The following properties hold:

\begin{description}
	\item[$(g_1)$] $\min\left\{s^{\delta},s^{g_0}\right\}g(t) \le g(st) \le \max\left\{s^{\delta},s^{g_0}\right\}g(t)$;
	
	\item[$(g_2)$] $G$ is a convex and $C^2$ function;
	
	\item[$(g_3)$] $\dfrac{sg(s)}{g_0+1} \le G(s) \le sg(s),$ for any $s\ge 0$;

    \item[$(G_1)$]  Defining $\xi_0(t)\coloneqq \min\{t^{\delta+1},t^{g_0+1}\}$ and $\xi_1(t)\coloneqq \max\{t^{\delta+1},t^{g_0+1}\}$, for any $s,t\ge0$ and $u\in L_G(U)$ hold
        $$\xi_0(s)\dfrac{G(t)}{g_0+1} \le G(st) \le (g_0+1)\xi_1(s)G(t), \quad \mbox{and} \quad \xi_0(|u|_{L_G}) \le \int_{U}G(|u|)\ dx \le \xi_1(|u|_{L_G});$$

    \item[$(G_2)$]  $G(s)+G(t) \le G(s+t) \le 2^{g_0}(g_0+1)\big(G(s) + G(t)\big)$ for any $s,t\ge0$;

    \item[$(G_3)$] $1<\delta+1 \le \dfrac{g(s)s}{G(s)} \le g_0+1$;
\end{description}
\end{statement}

Therefore, $G$ is an $\mathcal{N}$-function that satisfies the $\Delta_2$-condition. Furthermore, a crucial implication of \eqref{Ga} is that
\begin{equation}\label{UE}
	\min\{\delta,1\} \dfrac{g(|x|)}{|x|} |\xi|^2 \le \mathfrak{a}_{i,j} \xi_i \xi_j \le \max\{g_0,1\} \dfrac{g(|x|)}{|x|} |\xi|^2,
\end{equation}
for any $x,\xi \in \mathbb{R}^N$ and $i,j =1,...,n$, where $\mathfrak{a}_{i,j} = \frac{\partial \mathfrak{A}_i}{\partial x_j}$ with $\mathfrak{A}(x)=g(|x|)\frac{x}{|x|}$. This inequality means that the equation
\begin{equation}\label{G-harmonic}
	-\mbox{div} \left(g(|\nabla {\bf v}|)\dfrac{\nabla {\bf v}}{|\nabla {\bf v}|}\right) = 0, \quad \mbox{in} \ U,
\end{equation}
is uniformly elliptic for $\frac{g(|x|)}{|x|}$ bounded and bounded away from zero, for any domain $U\subset \mathbb{R}^n$. Moreover, if ${\bf v} \in W^{1,G}(U;\mathbb{R}^m)$ satisfies \eqref{G-harmonic}, we say that ${\bf v}$ is $g$-harmonic in $U$.

In addition to the previously established properties, we also derive the following crucial inequality: if ${\bf u} = (u_1, \dots, u_m)$ belongs to $L_G(U; \mathbb{R}^m)$, then there exists a constant $c=c(\delta,g_0)>0$ such that
    \begin{equation}\label{equivprop}
        \dfrac{1}{c}\intav{U} G(|\nabla {\bf u}|) \ dx \le  \intav{U} \sum_{i=1}^m G(|\nabla u_i|) \ dx \le c \intav{U} G(|\nabla {\bf u}|)\ dx.
    \end{equation}
\bigskip

\subsection{Auxiliary results}

Let us denote
\[
\Omega_{x_0,\varrho} \coloneqq \Omega\cap B_\varrho(x_0)\quad\text{and}\quad (u)_{x_0,\varrho} \coloneqq \frac{1}{|\Omega_{x_0,\varrho}|}\int_{\Omega_{x_0,\varrho}}u(x)\:dx = \intav{\Omega_{x_0,\varrho}} u(x)dx.
\]
For our purposes, the analog of the $L^p$ norm and Campanato seminorm (see, for instance, Giusti's book \cite{Giusti2003}) will be given, for $\lambda\geq0$, by
\begin{equation}\label{DefPhi}
  \Phi(u)\coloneqq \int_\Omega \mathrm{G}(|u|)\:dx,\quad \Phi_\lambda(u) \coloneqq \sup_{\substack{x_0\in\Omega \\ \varrho>0}}\varrho^{-\lambda}\int_{\Omega_{x_0,\varrho}}\mathrm{G}(|u-u_{x_0,\varrho}|)\:dx
\end{equation}
These quantities are not homogeneous and hence not a norm/seminorm, but they capture the essential properties that will be used to prove the desired continuity.

We will work with functions satisfying
\begin{equation}\label{eq.campanato}
\Phi(u)+\Phi_\lambda(u)\leq \mathrm{C}_{\Phi}
\end{equation}
for some constant $\mathrm{C}_{\Phi}$.

\begin{proposition}[{\cite[Proposition 4.2]{daSSV}}]
If $u$ satisfies \eqref{eq.campanato} and $\lambda>n$ then there exists $\mathrm{C}=\mathrm{C}(n,g_0,\lambda)>0$ such that
\[
\sup_{\substack{x,y\in\Omega \\ x\neq y}}\frac{|u(x)-u(y)|}{|x-y|^\gamma}\leq \mathrm{C}\mathrm{G}^{-1}\left(\frac{2^{g_0}}{|B_1|}\Phi_\lambda(u)\right)
\]
with $\gamma=\frac{\lambda-n}{g_0+1}$.
\end{proposition}

The next theorem is a Campanato-type result in the context of scalar Orlicz modular.
\begin{theorem}[{\cite[Theorem 4.3]{daSSV}}]\label{campanato}
Let $u$ be a measurable function satisfying
\[
\int_\Omega \mathrm{G}(|u|)\:dx+\sup_{\substack{x_0\in\Omega \\ \varrho>0}}\varrho^{-\lambda}\inf_{\xi\in\mathbb{R}}\int_{\Omega_{x_0,\varrho}}\mathrm{G}(|u(x)-\xi|)\:dx\leq \mathrm{C}_0
\]
for some positive constant $\mathrm{C}_0$ and $\lambda>n$. Then, $u\in C^{0, \gamma}(\Omega)$ with $\gamma=\frac{\lambda-n}{g_0+1}$ and
\[
\sup_{\substack{x,y\in\Omega \\ x\neq y}}\frac{|u(x)-u(y)|}{|x-y|^\gamma}\leq \mathrm{C}\mathrm{G}^{-1}\left(\sup_{\substack{x_0\in\Omega \\ \varrho>0}}\varrho^{-\lambda}\inf_{\xi\in \mathbb{R}}\int_{\Omega_{x_0,\varrho}}\mathrm{G}(|u(x)-\xi|)dx\right)
\]
\end{theorem}

\bigskip

\begin{theorem}[{\cite[Theorem 4.1]{MascPapi96}}]\label{Harnack-Ineq}
Let $G \in \Delta_2 \cap \nabla_2$ and $u \in W^{1,G}(\Omega)$ be a local positive minimizer of $\mathcal{J}_G$. Then for every $R$ with $Q_R \subset \Omega$,
\[
\sup_{Q_R} u \leq \mathrm{c}({\verb"universal"}) \inf_{Q_R} u.
\]
\end{theorem}

%Moreover, by standard covering arguments, we have also:

%\begin{theorem}[{\cite[Theorem 4.2]{MascPapi96}}]Assume $\Omega$ is a bounded connected open set of $\mathbb{R}^n$ and let $E \Subset \Omega$. Let $G \in \Delta_2 \cap \nabla_2$ and $u \in W^{1} L^G(\Omega)$ be a positive local minimizer of $\mathcal{J}_G$. Then there exists a constant $\mathrm{C}(E, m, \Omega)$ such that
%\[
%\sup_{x \in E} u(x) \leq \mathrm{C}(E, m, \Omega) \inf_{x \in E} u(x).
%\]
%\end{theorem}

The following result is well known in the scalar setting (see, for example, \cite[Lemma 2.5]{MW08} and \cite[Lemma 3.1]{BM}). However, as we could not find a corresponding version in the vectorial case, we included the proof here for completeness and to ensure clarity for the reader. 
\begin{lemma}\label{lemreg}
    Consider $G\in \mathcal{G}(\delta,g_0)$. Let ${\bf u} \in W_{loc}^{1,1}(U;\mathbb{R}^m) \cap L^1(U;\mathbb{R}^m)$, and $0<\alpha \le \frac{\delta}{g_0}\le1$. Suppose $U^{\prime} \Subset U$ such that
        $$
        \intav{B_r(x_0)} G\big(|\nabla {\bf u}(x)|\big)\ dx \le \mathfrak{L} r^{\alpha-1},
        $$
    for any $x_0 \in U^{\prime}$, and $0< r \le R_0 \le \mbox{dist}(U^{\prime}, \partial U)$. Then, ${\bf u} \in C^{0,\alpha}(U^{\prime};\mathbb{R}^m)$, with the following estimate
        $$[{\bf u} ]_{C^{0,\alpha}(U^{\prime};\mathbb{R}^m)} \le \mathrm{C} G^{-1}(\mathfrak{L}),$$
    with $\mathrm{C}=\mathrm{C}(\alpha,n,m,\delta,g_0)>0$.
\end{lemma}
\begin{proof}
    First of all, by \eqref{equivprop} we get
        \begin{equation}\label{27}
            \intav{B_r(x_0)} \sum_{i=1}^m G\big(|\nabla u_i(x)|\big) \le \mathrm{C} \mathfrak{L}r^{\alpha-1},
        \end{equation}
    for some $\mathrm{C}>0$. Next, by Poincar\'{e}'s inequality (see \eqref{Poincare}) and $(G_1)$, for any $i \in \{1,\dots,m\}$ there exists a constant $c_1=c_1(g_0)>0$ such 
        \begin{eqnarray*}
          \dfrac{c_1}{r^{\delta+1}} \intav{B_r(x_0)} G(|u_i(x) - (u_i)_{x_0,r}|) \ dx &\le& \intav{B_r(x_0)} G\left( \tau\dfrac{|u_i(x) - (u_i)_{x_0,r}|}{2r}\right)\ dx \\
          &\le& \intav{B_r(x_0)} G\big(|\nabla u_i(x)|\big)\ dx, \quad x \in U^{\prime}, \ r \le R_0.
        \end{eqnarray*}
    Then, since $0<\alpha \le \frac{\delta}{g_0}$, \eqref{27} implies
        $$\intav{B_r(x_0)} \sum_{i=1}^m G(|u_i(x) - (u_i)_{x_0,r}|) \ dx \le \dfrac{r^{\delta+1}}{c_1} \intav{B_r(x_0)} \sum_{i=1}^m G(|\nabla u_i(x)|) \ dx \le C\mathfrak{L} r^{\alpha(g_0+1)}.$$
    Therefore, by invoking Theorem \ref{campanato} we get
        $$
        \sum_{i=1}^m\ [u_i]_{C^{0,\alpha}(U^{\prime})} \le CG^{-1}(\mathfrak{L}),$$
    and the result follows.
\end{proof}

%{\color{red}\begin{remark}
%    Note that in the previous result, we restricted on $\alpha$, namely $0 < \alpha \le \frac{\delta}{g_0} \le 1$. It is worth emphasizing that in the $p$-Laplacian case -- i.e., when $G(t) = t^p$ -- this condition is automatically satisfied, since $\delta = p = g_0$. This leads us to a natural question: Is it possible to remove this condition and still obtain the same result?
%\end{remark}}

Now, for a given $\mathcal{N}$-function $G\in \mathcal{G}(\delta,g_0)$, define the map $\mathbf{V}_G:\mathbb{R}^n \to \mathbb{R}^n$ by
    \begin{equation}\label{defVg}
        \mathbf{V}_G(z) \coloneqq \left(\dfrac{G^{\prime}(|z|)}{|z|}\right)^{\frac{1}{2}}z.
    \end{equation}
Furthermore, we have the following monotonicity property (see, for instance, \cite[Lemma 2.5]{BB} in a double-phase structure)
    \begin{equation}\label{SSMI}
        \mathrm{C}|\mathbf{V}_G(z_1) - \mathbf{V}_G(z_2)|^2 + \dfrac{G^{\prime}(|z_1|)}{|z_1|} \ z_1\cdot (z_2 - z_1) \le G(|z_2|) - G(|z_1|),
    \end{equation}
for some $\mathrm{C}>0$. In addition, by direct computations, using Jensen's inequality, for any $v,w \in L_G(B_r(x_0))$ and $L \in \mathbb{R}$ holds
    \begin{equation}\label{211}
        \intav{B_r(x_0)} \big|\mathbf{V}_G(\nabla v) - \big(\mathbf{V}_G(\nabla v)\big)_{x_0,r}\big|^2 dx \le 4 \intav{B_r(x_0)} \big|\mathbf{V}_G(\nabla v) - L\big|^2 dx,
    \end{equation}
and there exists a constant $\mathrm{C}=\mathrm{C}(g_0)>0$ such that
    \begin{equation}\label{212}
        \intav{B_r(x_0)} \big|\big(\mathbf{V}_G(\nabla v)\big)_{x_0,r} - \big(\mathbf{V}_G(\nabla w)\big)_{x_0,r}\big|^2 dx \le \mathrm{C} \intav{B_r(x_0)} \big|\mathbf{V}_G(\nabla v) - \mathbf{V}_G(\nabla w)\big|^2 dx.
    \end{equation}

The next lemma deals with the excess functional and will be used to prove the H\"{o}lder regularity of $(\kappa,\beta)$-almost-minimizers of $\mathcal{J}_G$.

\begin{lemma}[{see \cite[Lemma 2.10 and Theorem 6.4]{DSV}}]\label{L22}
    Let $G\in \mathcal{G}(\delta,g_0)$, and let $\mathbf{V}_G$ be the function defined in \eqref{defVg}.
        \begin{enumerate}[$(a)$]
            \item The map $\mathbf{V}_G$ is invertible, and there exists $\gamma>0$, which depends only on $\delta$ and $g_0$, such that both $\mathbf{V}_G$ and its inverse $\mathbf{V}_G^{-1}$ are $\gamma$-H\"{o}lder continuous on $\mathbb{R}^n$;

            \item Let $v\in L_G(U)$ be a $g$-harmonic function in $U$. Then, there exists $\sigma=\sigma(n,\delta,g_0)\in(0,1)$ such that for any balls $B \Subset U$ and any $\tau \in (0,1)$ holds
        $$\intav{\tau B} \big|\mathbf{V}_G(\nabla v) - \big(\mathbf{V}_G(\nabla v)\big)_{\tau B}\big|^2 \ dx \le \mathrm{C} \tau^\sigma \intav{B} \big|\mathbf{V}_G(\nabla v) - \big(\mathbf{V}_G(\nabla v)\big)_{B}\big|^2 \ dx,$$
        \end{enumerate}
    where $\mathrm{C}=\mathrm{C}(n,\delta,g_0)>0$.
\end{lemma}

In the next lemma, as well as throughout the paper, we shall refer to $u^*$, a scalar function, as the $g$-harmonic replacement of $u$ in $B_\mathcal{S}(kx_0)\subset U$, meaning the unique $g$-harmonic function in $B_\mathcal{S}(kx_0)$ with the same trace as $u$ on $\partial B_\mathcal{S}(kx_0)$.

\begin{lemma}\label{controlVg}
    Assume that $G\in \mathcal{G}(\delta,g_0)$. Consider ${\bf u}=(u_1,\dots, u_m)$ a $(\kappa,\beta)$-almost-minimizer of $\mathcal{J}_G$ in $U$, and let $x_0 \in U$ and $B_\mathcal{S}(kx_0) \Subset U$. If $u_i^*$ is the $g$-harmonic replacement of $u_i$ in $B_\mathcal{S}(kx_0)$, for each $i \in \{1,\dots,m\}$, then
        $$
        \int_{B_\mathcal{S}(kx_0)} \sum_{i=1}^m \big|\mathbf{V}_G(\nabla u_i) - \mathbf{V}_G(\nabla u_i^*)\big|^2 \ dx \le \mathrm{C}\kappa s^\beta \int_{B_\mathcal{S}(kx_0)} \sum_{i=1}^m G(|\nabla u_i|) \ dx + \mathrm{C}\lambda s^n,
        $$
    where $\mathbf{V}_G$ is defined in \eqref{defVg}, and $\mathrm{C}=\mathrm{C}(n,\delta,g_0)>0$.
\end{lemma}
\begin{proof}
    Initially, observe that $\displaystyle \int_{B_\mathcal{S}(kx_0)} \frac{G^{\prime}\big(|\nabla u_i^*|\big)}{|\nabla u_i^*|} \nabla u_i^* \cdot \nabla \phi\ dx = 0$ for any function $\phi \in W_0^{1,G}(B_\mathcal{S}(kx_0))$, and that $u_i^*$ coincides with $u_i$ on $\partial B_\mathcal{S}(kx_0)$ in the trace sense. Thus, taking $\phi = u_i - u_i^*$ and invoking \eqref{SSMI}, it follows that
        \begin{eqnarray}\label{part1}
            \int_{B_\mathcal{S}(kx_0)} \sum_{i=1}^m \big|\mathbf{V}_G(\nabla u_i) - \mathbf{V}_G(\nabla u_i^*)\big|^2 \ dx &=& \int_{B_\mathcal{S}(kx_0)} \sum_{i=1}^m \big|\mathbf{V}_G(\nabla u_i) - \mathbf{V}_G(\nabla u_i^*)\big|^2 \ dx \nonumber \\
            && + \int_{B_\mathcal{S}(kx_0)} \sum_{i=1}^m\dfrac{G^{\prime}(|\nabla u_i^*|)}{|\nabla u_i^*|} \nabla u_i^* \cdot \nabla (u_i - u_i^*) \ dx \nonumber \\
            &\le& C \int_{B_\mathcal{S}(kx_0)} \sum_{i=1}^m \Big(G(|\nabla u_i|) - G(|\nabla u_i^*|)\Big) \ dx.
        \end{eqnarray}
    Now, set ${\bf v} = (u_1^*, \dots, u_m^*)$ and use it in the definition of almost minimizer of ${\bf u}$ to get
        \begin{eqnarray}\label{part2}
            \int_{B_\mathcal{S}(kx_0)} \sum_{i=1}^m \Big( G(|\nabla u_i|) - G(|\nabla u_i^*|) \Big) \ dx &\le& \kappa s^\beta \int_{B_\mathcal{S}(kx_0)} \sum_{i=1}^m G(|\nabla u_i^*|) \ dx + \mathrm{C}\lambda s^n \nonumber \\
            &\le& \kappa s^\beta \int_{B_\mathcal{S}(kx_0)} \sum_{i=1}^m G(|\nabla u_i|) \ dx + \mathrm{C}\lambda s^n,
        \end{eqnarray}
    where in the last inequality we use the minimality of $u_i^*$ in the functional $\displaystyle \mathcal{I}(v) = \int G(|\nabla v|)dx$. Finally, the result follows by combining \eqref{part1} and \eqref{part2}.
\end{proof}

\section{H\"{o}lder regularity estimates}\label{Sec3}

This section is devoted to proving the initial regularity result for almost-minimizers of the functional $\mathcal{J}_G$. More precisely, we show that such functions exhibit local $C^{0,\alpha}$-regularity for every $\alpha \in (0,\delta/g_0]$. Additionally, we establish that the gradient of such functions is (locally)  H\"{o}lder continuous away from the free boundary.

%\begin{theorem}\label{Holdereg}
%    Let $G$ be an $\mathcal{N}$-function satisfying \eqref{Ga}. Consider ${\bf u}=(u_1,\dots, u_m)$ a $(\kappa,\beta)$-almost-minimizer of $\mathcal{J}_G$ in $\Omega$, with some positive constant $\kappa \le \kappa_0$ and exponent $\beta>0$. Then, ${\bf u} \in C_{loc}^{0,\alpha}(\Omega; \mathbb{R}^m)$, for any $0<\alpha \le \frac{\delta}{g_0}$. More precisely, for any $\Omega' \Subset \Omega$, there exists a constant $C=C(n,m,\beta,\kappa_0,\delta,g_0)>0$ such that
%       $$\| {\bf u}\|_{C^{0,\alpha}(\Omega';\mathbb{R}^m)} \le CG^{-1}\left( \sum_{i=1}^m \xi_1\Big(|\nabla u_i|_{L_G(\Omega)}\Big)+\lambda\right).$$
%\end{theorem}
\begin{proof}[Proof of Theorem \ref{Holdereg}]
    In order to use the Lemma \ref{lemreg}, let us estimate the integral $\displaystyle \intav{B_r(x_0)} G(|\nabla {\bf u}|) \ dx$, where $B_r(x_0) \Subset \Omega$, with $r \le 1$. To do so, as before, for each $i \in \{1,\dots,m\}$ set $u_i^*$ the $g$-harmonic replacement of $u_i$ in $B_r(x_0)$. 
    Moreover, for any $u_i^*$ we have the following estimate (see, for instance, \cite[Lemma 5.8]{DSV})
        \begin{equation}\label{31}
            \int_{B_{s_1}(x_0)} G(|\nabla u_i^*|) \ dx \le \mathrm{C} \left(\dfrac{s_1}{s_2}\right)^n\int_{B_{s_2}(x_0)}G(|\nabla u_i^*|) \ dx,
        \end{equation}
    for a universal $\mathrm{C}>0$, and any $0<s_1 < s_2 \le 1$. 
    
    Now, set $\tau \in (0,1)$, and by definition of $\mathbf{V}_G$ in \eqref{defVg} and $(G_3)$ we have
        \begin{eqnarray*}
            \int_{B_{\tau r}(x_0)} \sum_{i=1}^m \big|\mathbf{V}_G(\nabla u_i) - \mathbf{V}_G(\nabla u_i^*)\big|^2 dx &\ge& \int_{B_{\tau r}(x_0)} \sum_{i=1}^m \Big(G^{\prime}(|\nabla u_i|)|\nabla u_i| - G^{\prime}(|\nabla u_i^*|)|\nabla u_i^*|\Big) dx \\ 
            &\ge& (\delta+1) \int_{B_{\tau r}(x_0)} \sum_{i=1}^m G(|\nabla u_i|)dx  \\
            &&- (g_0+1) \int_{B_{\tau r}(x_0)} \sum_{i=1}^m G(|\nabla u_i^*|)dx,
        \end{eqnarray*}
    which implies
        \begin{eqnarray*}
            \int_{B_{\tau r}(x_0)} \sum_{i=1}^m G(|\nabla u_i|)dx &\le& \dfrac{g_0+1}{\delta+1} \left(\int_{B_{\tau r}(x_0)} \sum_{i=1}^m \big|\mathbf{V}_G(\nabla u_i) - \mathbf{V}_G(\nabla u_i^*)\big|^2 dx \right. \\
            && \left. \hspace{3.5cm}+ \int_{B_{\tau r}(x_0)} \sum_{i=1}^m G(|\nabla u_i^*|)dx\right).
        \end{eqnarray*}
    Thus, Lemma \ref{controlVg}, \eqref{31}, and the minimality of $u_i^*$ yield
        \begin{eqnarray}\label{32}
            \int_{B_{\tau r}(x_0)} \sum_{i=1}^m G(|\nabla u_i|) dx &\le&  \mathrm{C}\kappa r^\beta \int_{B_r(x_0)} \sum_{i=1}^m G(|\nabla u_i|) dx + \mathrm{C}\lambda r^n + \mathrm{C} \tau^n \int_{B_r(x_0)}\sum_{i=1}^mG(|\nabla u_i^*|)dx \nonumber \\
            &\le& \mathrm{C}\big(\kappa r^\beta + \tau^n\big)\int_{B_r(x_0)}\sum_{i=1}^mG(|\nabla u_i|)dx + \mathrm{C} \lambda r^n \nonumber \\
            &\le& \tau^{n+\alpha-1} \Big(r^\beta c_*\kappa_0\tau^{1-\alpha-n} + c_*\tau^{1-\alpha}\Big)\int_{B_r(x_0)}\sum_{i=1}^mG(|\nabla u_i|)dx + \mathrm{C} \lambda r^n,
        \end{eqnarray}
    for any $\alpha\in(0,1)$, and some $c_*>1$ (without loss of generality). Now, if we get $\tau\in(0,1)$ such that $c_*\tau^{1-\alpha} \le \frac{1}{2}$, and $0<r \le R_0 \le 1$ such that $R_0^\beta c_*\kappa_0 \tau^{1-\alpha-n} \le \frac{1}{2}$, \eqref{32} becomes
        $$
        \int_{B_{\tau r}(x_0)} \sum_{i=1}^m G(|\nabla u_i|) dx \le \tau^{n+\alpha-1} \int_{B_r(x_0)}\sum_{i=1}^mG(|\nabla u_i|)dx + \mathrm{C} \lambda r^n.$$
    Notice that, by the previous inequality, changing $\tau>0$ if necessary
        \begin{eqnarray*}
            \int_{B_{\tau^2 r}(x_0)} \sum_{i=1}^m G(|\nabla u_i|) dx &\le& \tau^{n+\alpha-1} \int_{B_{\tau r}(x_0)}\sum_{i=1}^mG(|\nabla u_i|)dx + \mathrm{C} \lambda (\tau r)^n \\
            &\le& \tau^{n+\alpha-1} \left(\tau^{n+\alpha-1} \int_{B_r(x_0)}\sum_{i=1}^mG(|\nabla u_i|)dx + \mathrm{C} \lambda r^n\right) + \mathrm{C}\lambda r^n\tau^n \\
            &=& \tau^{2(n+\alpha-1)} \int_{B_r(x_0)}\sum_{i=1}^mG(|\nabla u_i|)dx + \mathrm{C} \lambda r^n \big(\tau^{n+\alpha-1}+\tau^n\big).
        \end{eqnarray*}
    Inductively, for any $k \in \mathbb{N}$, we get
        \begin{equation}\label{path}
            \int_{B_{\tau^k r}(x_0)} \sum_{i=1}^m G(|\nabla u_i|) dx \le \tau^{k(n+\alpha-1)} \int_{B_r(x_0)}\sum_{i=1}^mG(|\nabla u_i|)dx + \mathrm{C}\lambda r^n \dfrac{\tau^{k(n+\alpha-1)}-\tau^{kn}}{\tau^{n+\alpha-1}-\tau^n}.
        \end{equation}
    Next, by the choices of $\tau$ and $R_0$ we get
        $$\tau^{n+\alpha-1} - \tau^n > 2R_0^\beta c_*\kappa_0 - \left(\dfrac{1}{2c_*}\right)^{\frac{n}{1-\alpha}} = \mathrm{C}(R_0,\beta,n,\alpha)\left(c_*-\left(\dfrac{1}{c_*}\right)^{\frac{n}{1-\alpha}}\right)>0,$$
    since $c_*>1$. Thus, we may estimate \eqref{path} by 
        \begin{equation}\label{33}
            \int_{B_{\tau^k r}(x_0)} \sum_{i=1}^m G(|\nabla u_i|) dx \le \tau^{k(n+\alpha-1)} \int_{B_r(x_0)}\sum_{i=1}^mG(|\nabla u_i|)dx + \mathrm{C}\lambda \big(\tau^k r\big)^{n+\alpha-1}.
        \end{equation}
    Now, consider a natural number $k$ such that $\tau^{k+1}r \le s \le \tau^kr$. Then, \eqref{33} leads us to
        \begin{eqnarray}\label{34}
            \int_{B_s(x_0)} \sum_{i=1}^m G(|\nabla u_i|) \ dx &\le& \int_{B_{\tau^k r}(x_0)} \sum_{i=1}^m G(|\nabla u_i|)\ dx \nonumber \\
            &\le& \tau^{1-\alpha-n}\tau^{(k+1)(n+\alpha-1)} \int_{B_r(x_0)}\sum_{i=1}^mG(|\nabla u_i|)dx \nonumber \\
            && \hspace{2cm}+ \mathrm{C}\lambda\tau^{1-\alpha-n} \big(\tau^{k+1} r\big)^{n+\alpha-1} \nonumber \\
            &\le& \mathrm{C} \left(\dfrac{s}{r}\right)^{n+\alpha-1}\int_{B_r(x_0)}\sum_{i=1}^mG(|\nabla u_i|)dx + \mathrm{C}\lambda s^{n+\alpha-1},
        \end{eqnarray}
    for any $0<s<r\le R_0 \le 1$, with $R_0^\beta c_*\kappa_0 \tau^{1-\alpha-n} \le \frac{1}{2}$. Now, we need to prove a similar inequality, but without this restriction, which means for any $0<s<r\le 1$. Suppose first $0<s\le R_0 <r \le 1$. In this case, by \eqref{34}
        \begin{eqnarray*}
            \int_{B_s(x_0)} \sum_{i=1}^m G(|\nabla u_i|) \ dx &\le& \mathrm{C} \left(\dfrac{s}{R_0}\right)^{n+\alpha-1}\int_{B_{R_0}(x_0)}\sum_{i=1}^mG(|\nabla u_i|)dx + \mathrm{C}\lambda s^{n+\alpha-1} \\
            &\le& \mathrm{C} \left(\dfrac{s}{r}\right)^{n+\alpha-1} \left(\dfrac{r}{R_0}\right)^{n+\alpha-1}\int_{B_r(x_0)}\sum_{i=1}^mG(|\nabla u_i|)dx + \mathrm{C}\lambda s^{n+\alpha-1} \\
            &\le& \dfrac{\mathrm{C}}{R_0^{n+\alpha-1}} \left(\dfrac{s}{r}\right)^{n+\alpha-1}\int_{B_r(x_0)}\sum_{i=1}^mG(|\nabla u_i|)dx + \mathrm{C}\lambda s^{n+\alpha-1}
        \end{eqnarray*}
    While, if $0<R_0<s <r\le1$ we have
        \begin{eqnarray*}
            \int_{B_s(x_0)} \sum_{i=1}^m G(|\nabla u_i|) \ dx &\le& \left(\dfrac{s}{R_0}\right)^{n+\alpha-1}\int_{B_s(x_0)} \sum_{i=1}^m G(|\nabla u_i|) \ dx + \mathrm{C}\lambda s^{n+\alpha-1} \\
            &\le& \dfrac{1}{R_0^{n+\alpha-1}} \left(\dfrac{s}{r}\right)^{n+\alpha-1}\int_{B_r(x_0)} \sum_{i=1}^m G(|\nabla u_i|) \ dx + \mathrm{C}\lambda s^{n+\alpha-1}.
        \end{eqnarray*}
    In either case, inequality \eqref{34} is thereby established for all radii satisfying $0<s<r\le1$, as required. As a consequence, applying a standard covering argument and $(G_1)$, we obtain that for any $\alpha \in (0,1)$
        \begin{eqnarray}\label{hf}
            \intav{B_r(x_0)} \sum_{i=1}^m G(|\nabla u_i|) \ dx &\le& \mathrm{C} \left(\int_{\Omega} \sum_{i=1}^mG\big(|\nabla u_i|\big) \ dx + \lambda\right) r^{\alpha-1} \nonumber \\
            &\le& \mathrm{C} \left( \sum_{i=1}^m \xi_1\Big(|\nabla u_i|_{L_G(\Omega)}\Big) + \lambda \right)r^{\alpha-1},
        \end{eqnarray}
    from which combined with \eqref{equivprop} we deduce
        $$\intav{B_r(x_0)} G(|\nabla {\bf u}|)\ dx \le \mathrm{C} \left( \sum_{i=1}^m \xi_1\Big(|\nabla u_i|_{L_G(\Omega)}\Big) + \lambda \right)r^{\alpha-1},$$
    and the conclusion follows by choosing $\alpha \in (0,\delta/g_0]$ and invoking Lemma \ref{lemreg}.
\end{proof}

\subsection{\texorpdfstring{$C^{1, \alpha}$}{C1,α} estimates far away from the free boundary}

In this subsection, we obtain H\"{o}lder gradient estimates for $(\kappa,\beta)$-almost-minimizers of \eqref{DefFunctional} far away from the free boundary. 

For the next result, we recall that, given ${\bf u} = (u_1,\dots,u_m)$, we define $\displaystyle P_{\bf u} \coloneqq \bigcup_{i=1}^m \{u_i>0\}$.

\begin{proof}[Proof of Theorem \ref{gradHold}]
    The main idea is to prove that $\mathbf{V}_G(\nabla u_i)$ is locally H\"{o}lder continuous in $P_{\bf u}$, where $\mathbf{V}_G$ is defined in \eqref{defVg}. Then, we apply Lemma \ref{L22} $(a)$.

    First, consider $B_r(x_0) \Subset P_{\bf u}$, and similar to the previous proof, for any $i \in \{1,\dots,m\}$, define $u_i^*$ as the $g$-harmonic replacement of $u_i$ in $B_r(x_0)$. Thus, by Lemma \ref{controlVg}
        \begin{equation}\label{36}
            \intav{B_r(x_0)} \sum_{i=1}^m \big|\mathbf{V}_G(\nabla u_i) - \mathbf{V}_G(\nabla u_i^*)\big|^2 \ dx \le \mathrm{C}\kappa r^\beta\left( \intav{B_r(x_0)} \sum_{i=1}^m G(|\nabla u_i|) \ dx + \mathrm{C}\lambda r^n\right).
        \end{equation}
    Mimicking the proof of \eqref{hf} in Theorem \ref{Holdereg}, we get
        \begin{equation}\label{sumilux}
            \intav{B_r(x_0)} \sum_{i=1}^m G(|\nabla u_i|) \ dx \le \mathrm{C} \left(\sum_{i=1}^m \int_\Omega G \big(|\nabla u_i|\big)\ dx + \lambda\right) r^{\alpha-1},
        \end{equation}
    where $C=C(\delta,g_0,n,\kappa_0,\beta)>0$, and $\alpha\in(0,1)$, which plugged in \eqref{36} gives us
        \begin{equation}\label{I}
            \intav{B_r(x_0)} \sum_{i=1}^m \big|\mathbf{V}_G(\nabla u_i) - \mathbf{V}_G(\nabla u_i^*)\big|^2 \le \mathrm{C}\kappa r^\beta\left(\sum_{i=1}^m \int_\Omega G \big(|\nabla u_i|\big)\ dx + \lambda\right) r^{\alpha-1}.
        \end{equation}
    Next, Lemma \ref{L22} $(b)$, \eqref{211} with $L=0$, and $(G_3)$ lead us to
        \begin{eqnarray}\label{II}
            \intav{B_s(x_0)}\sum_{i=1}^m\big|\mathbf{V}_G(\nabla u_i^*) - \big(\mathbf{V}_G(\nabla u_i^*)\big)_{x_0,s}\big|^2 dx &\le& \mathrm{C} \left(\dfrac{s}{r}\right)^\sigma \intav{B_r(x_0)}\sum_{i=1}^m G(|\nabla u_i^*|) \ dx \nonumber \\
            &\stackrel{\eqref{sumilux}}{\le}& \mathrm{C} \left(\dfrac{s}{r}\right)^\sigma \left(\sum_{i=1}^m \int_\Omega G \big(|\nabla u_i|\big)\ dx + \lambda\right) r^{\alpha-1}, 
        \end{eqnarray}
    for any $0<s< r$, and some $\sigma>0$, we also used the minimality of $u_i^*$ in the last inequality. Finally, by \eqref{212} and \eqref{I} we obtain
        \begin{eqnarray}\label{III}
            \intav{B_s(x_0)}\sum_{i=1}^m\big|\big(\mathbf{V}_G(\nabla u_i^*)\big)_{x_0,s} - \big(\mathbf{V}_G(\nabla u_i)\big)_{x_0,s}\big|^2 dx &\le& \mathrm{C} \intav{B_s(x_0)}\sum_{i=1}^m\big|\mathbf{V}_G(\nabla u_i^*) - \mathbf{V}_G(\nabla u_i)\big|^2 dx \nonumber \\
            &\le& \mathrm{C}\kappa r^\beta \left(\sum_{i=1}^m \int_\Omega G \big(|\nabla u_i|\big)\ dx + \lambda\right) r^{\alpha-1}.
        \end{eqnarray}    
    
    Thus, by combining \eqref{I}-\eqref{III}, we get
        \begin{eqnarray}\label{campanatoVG}
            \intav{B_s(x_0)}\sum_{i=1}^m\big|\mathbf{V}_G(\nabla u_i) - \big(\mathbf{V}_G(\nabla u_i)\big)_{x_0,s}\big|^2 dx &\le& \intav{B_s(x_0)}\sum_{i=1}^m\big|\mathbf{V}_G(\nabla u_i) - \mathbf{V}_G(\nabla u_i^*)\big|^2 dx \nonumber \\
            &&+ \intav{B_s(x_0)}\sum_{i=1}^m\big|\mathbf{V}_G(\nabla u_i^*) - \big(\mathbf{V}_G(\nabla u_i^*)\big)_{x_0,s}\big|^2 dx \nonumber \\
            &&+ \intav{B_s(x_0)}\sum_{i=1}^m\big|\big(\mathbf{V}_G(\nabla u_i^*)\big)_{x_0,s} - \big(\mathbf{V}_G(\nabla u_i)\big)_{x_0,s}\big|^2 dx \nonumber \\
            &\le& \mathrm{C} \left(\sum_{i=1}^m \int_\Omega G \big(|\nabla u_i|\big)\ dx + \lambda\right) \left( \left(\dfrac{r}{s}\right)^n r^{\beta+\alpha-1} + \left(\dfrac{s}{r}\right)^\sigma r^{\alpha-1}\right). \nonumber\\
            &\le& \mathrm{C} \left(\sum_{i=1}^m \int_\Omega G \big(|\nabla u_i|\big)\ dx + \lambda\right) s^{\gamma},
        \end{eqnarray}
    where $\gamma = \frac{1-\alpha}{1+\theta} \in (0,1)$. To derive the last inequality, set $s=r^{1+\theta}$ with $\theta \coloneqq \frac{\beta}{\sigma+n}>0$, and observe that              \begin{eqnarray}\label{alphatogamma}
            \left(\dfrac{r}{s}\right)^n r^{\beta+\alpha-1} + \left(\dfrac{s}{r}\right)^\sigma r^{\alpha-1} &=& s^{\frac{-\theta n + \beta+\alpha-1}{1+\theta}} + s^{\frac{\theta\sigma + \alpha-1}{1+\theta}} \nonumber\\
            &=& 2s^{\frac{\theta\sigma+\alpha-1}{1+\theta}}
        \end{eqnarray}
    where we used the identity $-\theta n +\beta = \theta\sigma$. Now, assume that $\theta\sigma<2$. Since, by hypothesis, $g_0 = \delta + \varepsilon$ for some small $\varepsilon>0$ choose
        $$0<\varepsilon < (2-\theta\sigma)\delta\theta\sigma,$$
    which implies that
        $$\dfrac{\delta}{g_0} = \dfrac{\delta}{\delta+\varepsilon}>1-\dfrac{\theta\sigma}{2}.$$
    Hence, we may choose $\alpha$ such that $1-\frac{\theta\sigma}{2}<\alpha<\frac{\delta}{g_0}$, and thus obtain
        $$\dfrac{1-\alpha}{1+\theta}<\dfrac{\theta\sigma+\alpha-1}{1+\theta}.$$
    In the case where $\theta\sigma \ge 2$, the conclusion still holds without additional assumptions. Since $s \in (0,1)$, this inequality, combined with \eqref{alphatogamma}, yields the desired estimate \eqref{campanatoVG}.

    Therefore, following a similar argument as in the proof of Lemma \ref{lemreg}, and applying the classical Campanato space embedding theorem (see, e.g., \cite[Chapter III]{G}), we conclude that $\mathbf{V}_G(\nabla u_i)$ is locally $C^{0,\gamma}$ regular in $P_{\bf u}$.  The result then follows directly from Lemma \ref{L22} $(a)$.
\end{proof}

\begin{remark}\label{obs1}
    At the end of the previous proof, it becomes clear that the structural assumption $g_0 = \delta + \varepsilon$ is no longer necessary when the condition $\theta \sigma \ge 2$ is satisfied. This observation allows us to relax the assumptions on the ellipticity constants, provided that the exponent $\beta$ appearing in the definition of a $(\kappa, \beta)$-almost-minimizer is greater than the threshold $2\left(\dfrac{n}{\sigma} + 1\right)$, with $\sigma \in (0,1)$ given in the Lemma \ref{L22} $(b)$, which is equivalent to $\theta\sigma \ge 2$, by the definition of $\theta$. The next statement makes this refinement precise.
\end{remark}

\begin{corollary}\label{corgradhold}
    Suppose $G\in \mathcal{G}(\delta,g_0)$, and let ${\bf u} = (u_1, \dots, u_m)$ be a $(\kappa,\beta)$-almost-minimizer of $\mathcal{J}_G$ in $\Omega$, with some positive constant $\kappa \le \kappa_0$ and exponent $\beta \ge 2\left(\dfrac{n}{\sigma} + 1\right)$, where $0<\sigma<1$ is given in Lemma \ref{L22} $(b)$. Then ${\bf u}$ exhibits local $C^{1,\gamma}$-regularity in $P_{\bf u}$. More precisely, for any open set $\Omega^{\prime} \Subset P_{\bf u}$, there exists an exponent $\gamma = \gamma(\delta, g_0, n, \beta) > 0$ and a constant $\mathrm{C} = \mathrm{C}(\Omega^{\prime}, \delta, g_0, n, \kappa_0, \beta) > 0$ such that
        $$\|{\bf u}\|_{C^{1,\gamma}(\Omega^{\prime};\mathbb{R}^m)} \le \mathrm{C}\Big(\vert\nabla {\bf u}\vert_{L^G(\Omega;\mathbb{R}^m)} + \lambda\Big).$$
\end{corollary}

\section{Lipschitz regularity estimates}\label{Sec4}

In this section, we will prove that $(\kappa,\beta)$-almost-minimizers of the functional $\mathcal{J}_G$. To begin, let us prove a Caccioppoli-type inequality, which is one of the main steps towards our goal.

\begin{proposition}[{\bf Caccioppoli-type inequality}]\label{CaccioIne}
    Assume that $G\in \mathcal{G}(\delta,g_0)$, and let ${\bf u} = (u_1, \dots, u_m)$ be an $(\kappa,\beta)$-almost-minimizer of $\mathcal{J}_G$ in $\Omega$, with some positive constant $\kappa \le \kappa_0$ and exponent $\beta>0$. Then, for any concentric balls $B_s(x_0) \subset B_t(x_0) \Subset \Omega$, with $0<s<t<\infty$, and $\vartheta\in \mathbb{R}$, it holds
        $$
        \sum_{i=1}^m \int_{B_{s}(x_0)} G(|\nabla u_i|) \ dx \le \mathrm{C} \sum_{i=1}^m\left(\int_{B_t(x_0)} G\left(\dfrac{|u_i - \vartheta|}{t-s}\right) \ dx +\lambda |B_t(x_0)|\right), 
        $$
    for some universal constant $\mathrm{C} = \mathrm{C}(n,m,\kappa_0,\beta,\delta,g_0)>0$.
\end{proposition}
\begin{proof}
For simplicity of notation, for any $r>0$ let us consider $B_r(x_0)$ solely as $B_r$. Consider $\eta \in C_0^\infty(B_t)$, such that
    \begin{equation}\label{eta}
    \left\{
       \begin{array}{l}
           0\le \eta\le1 \\
           \eta \equiv 1 \ \mbox{in} \ B_s\\
           |\nabla \eta|\le \dfrac{2}{t-s}.
        \end{array}
    \right.\end{equation}
Define also $\varphi_i = \eta \big(u_i - \vartheta\big)$, which satisfies $\nabla \varphi_i =\nabla u_i$ in $B_s$. Then, for any $i \in \{1,\dots,m\}$
    \begin{eqnarray*}
        \int_{B_t} G(|\nabla \varphi_i|) \ dx &=& \int_{B_t} \Big( G(|\nabla u_i|) + \lambda\chi_{\{u_i>0\}}\Big)\ dx + \int_{B_t} \Big(G(|\nabla \varphi_i|) - G(|\nabla u_i|) -\lambda\chi_{\{u_i>0\}}\Big)\ dx \\
        &\le& \int_{B_t} \Big( G(|\nabla u_i|) + \lambda\chi_{\{u_i>0\}}\Big)\ dx + \int_{B_t\setminus B_s} \Big(G(|\nabla \varphi_i|) - G(|\nabla u_i|) -\lambda\chi_{\{u_i>0\}}\Big) \ dx \\
        &\le& \int_{B_t} \Big( G(|\nabla u_i|) + \lambda\chi_{\{u_i>0\}}\Big)\ dx + \int_{B_t\setminus B_s} \Big(G(|\nabla \varphi_i|) + G(|\nabla u_i|) +\lambda\Big) \ dx,
    \end{eqnarray*}
which leads us to
    \begin{eqnarray}\label{BsG}
        \int_{B_s}G(|\nabla u_i|) \ dx &=& \int_{B_t} G(|\nabla \varphi_i|) \ dx - \int_{B_t \setminus B_s} G(|\nabla \varphi_i|) \ dx \nonumber \\
        &\le& \int_{B_t} \Big( G(|\nabla u_i|) + \lambda\chi_{\{u_i>0\}}\Big)\ dx + \int_{B_t\setminus B_s} G(|\nabla u_i|) \ dx + \lambda|B_t|,
    \end{eqnarray}
for any $i \in \{1,\dots,m\}$. Now, let $v_i = u_i - \varphi_i = \vartheta + (1-\eta)(u_i-\vartheta)$, which clearly satisfies $v_i = \vartheta$ in $B_s$, and $v_i=u_i$ on $\partial B_t$. Thus, by the definition of $(\kappa,\beta)$-almost-minimizers, we have
    \begin{eqnarray*}
        \int_{B_t} \left(\sum_{i=1}^m G(|\nabla u_i|) + \lambda \chi_{\{|{\bf u}|>0\}}\right)\ dx &\le& \big(1+\kappa t^\beta\big) \int_{B_t} \left(\sum_{i=1}^m G(|\nabla v_i|) + \lambda \chi_{\{|{\bf v}|>0\}}\right)\ dx \\
        &\le& \big(1+\kappa t^\beta\big) \int_{B_t} \left(\sum_{i=1}^m G(|\nabla v_i|) + \lambda\right)\ dx.
    \end{eqnarray*}
Taking the sum from $1$ to $m$ in \eqref{BsG} and plugging the last inequality into it, we arrive at
    \begin{equation}\label{BsG2}
        \int_{B_s} \sum_{i=1}^m G(|\nabla u_i|) \ dx \le \int_{B_t\setminus B_s} \sum_{i=1}^m G(|\nabla u_i|) \ dx + \mathrm{C}\int_{B_t} \sum_{i=1}^m G(|\nabla v_i|) \ dx + \mathrm{C}\lambda|B_t|.
    \end{equation}
Next, by $(G_2)$ and \eqref{ine} we have
    \begin{eqnarray*}
        G(|\nabla v_i|) &=& G\big(|(1-\eta)\nabla u_i - (u_i - \vartheta)\nabla \eta|\big) \\
        &\le& (1-\eta) G(|\nabla u_i|) + \mathrm{C}G\left(\dfrac{|u_i - \vartheta|}{t-s}\right),
    \end{eqnarray*}
which combined with \eqref{BsG2} and \eqref{eta} yield
    \begin{eqnarray}\label{BsG3}
        \int_{B_s} \sum_{i=1}^m G(|\nabla u_i|) \ dx &\le& \int_{B_t\setminus B_s} \sum_{i=1}^m G(|\nabla u_i|) \ dx + \mathrm{C}\int_{B_t} \sum_{i=1}^m \left[(1-\eta) G(|\nabla u_i|) + G\left(\dfrac{|u_i - \vartheta|}{t-s}\right)\right] \ dx \nonumber \\
        &&\hspace{10cm}+ \mathrm{C}\lambda|B_t| \nonumber \\
        &=& \mathrm{C} \left( \int_{B_t\setminus B_s} \sum_{i=1}^m G(|\nabla u_i|) \ dx + \int_{B_t} \sum_{i=1}^mG\left(\dfrac{|u_i - \vartheta|}{t-s}\right) \ dx + \lambda |B_t|\right).
    \end{eqnarray}
    
We now use the ``hole-filling'' method; that is, we add to both sides the quantity
    $$
    \mathrm{C}\sum_{i=1}^m\int_{B_s} G(|\nabla u_i|)\ dx,$$
and divide by $C+1$. Then, we discover that
    \begin{equation}\label{induc1}
        \int_{B_s} \sum_{i=1}^m G(|\nabla u_i|) \ dx \le \gamma \int_{B_t} \sum_{i=1}^mG(|\nabla u_i|)\ dx + \int_{B_t} \sum_{i=1}^mG\left(\dfrac{|u_i - \vartheta|}{t-s}\right) \ dx + \lambda |B_t|,
    \end{equation}
where $\gamma \coloneqq \frac{\mathrm{C}}{\mathrm{C}+1}<1$, for any $0<s<t<\infty$ with $B_t\Subset\Omega$.

Now, fix $s<t$ and consider the sequence
    $$\rho_0 \coloneqq s \quad \mbox{and} \quad \rho_{j+1} = (1-\tau)\tau^j(t-s) + \rho_j, \quad j =0,1,2,\dots,$$
where $\tau \in (0,1)$ is to be chosen later. Notice that $\rho_j$ is a increasing sequence, and $\rho_j < t$ for any $j=0,1,2,\dots$. Then, applying \eqref{induc1}, we obtain
    \begin{eqnarray*}
        \int_{B_s} \sum_{i=1}^m G(|\nabla u_i|) \ dx &\le& \gamma \int_{B_{\rho_1}} \sum_{i=1}^m G(|\nabla u_i|) \ dx + \int_{B_{\rho_1}} \sum_{i=1}^mG\left(\dfrac{|u_i - \vartheta|}{(1-\tau)(t-s)}\right) \ dx + \lambda |B_t| \\
        &\le& \gamma \left(\gamma \int_{B_{\rho_2}} \sum_{i=1}^m G(|\nabla u_i|) \ dx + \int_{B_{\rho_2}} \sum_{i=1}^mG\left(\dfrac{|u_i - \vartheta|}{(1-\tau)\tau(t-s)}\right) \ dx + \lambda |B_t|\right) \\
        && \hspace{4cm}+ \int_{B_{\rho_1}} \sum_{i=1}^mG\left(\dfrac{|u_i - \vartheta|}{(1-\tau)(t-s)}\right) \ dx + \lambda |B_t| \\
        &\le& \gamma^2 \int_{B_{\rho_2}} \sum_{i=1}^m G(|\nabla u_i|) \ dx + \int_{B_t} \sum_{i=1}^mG\left(\dfrac{|u_i - \vartheta|}{(1-\tau)(t-s)}\right) \ dx \\
        && \hspace{3cm}+ \gamma\int_{B_t} \sum_{i=1}^mG\left(\dfrac{|u_i - \vartheta|}{(1-\tau)\tau(t-s)}\right) \ dx + \lambda|B_t|\Big(\gamma+1\Big)\\
        &\le&\gamma^2 \int_{B_{\rho_2}} \sum_{i=1}^m G(|\nabla u_i|) \ dx + \int_{B_t} \sum_{i=1}^mG\left(\dfrac{|u_i - \vartheta|}{(1-\tau)(t-s)}\right) \ dx \\
        && \hspace{2cm}+ \gamma(g_0+1)\tau^{-(g_0+1)}\int_{B_t} \sum_{i=1}^mG\left(\dfrac{|u_i - \vartheta|}{(1-\tau)(t-s)}\right) \ dx + \lambda|B_t|\Big(\gamma+1\Big)
    \end{eqnarray*}  
Inductively, we find
    \begin{eqnarray*}
         \int_{B_s} \sum_{i=1}^m G(|\nabla u_i|) \ dx &\le& \gamma^j\int_{B_{\rho_j}} \sum_{i=1}^m G(|\nabla u_i|) \ dx + (g_0+1) \sum_{k=0}^{j-1}(\gamma\tau^{-(g_0+1)})^k\int_{B_t} \sum_{i=1}^mG\left(\dfrac{|u_i - \vartheta|}{(1-\tau)(t-s)}\right) \ dx \\
        && \hspace{4cm}\lambda|B_t|\sum_{k=0}^{j-1}\gamma^k\\
        &\le& \gamma^j\int_{B_{\rho_j}} \sum_{i=1}^m G(|\nabla u_i|) \ dx + \dfrac{(g_0+1)^2}{(1-\tau)^{g_0+1}} \sum_{k=0}^{j-1}(\gamma\tau^{-(g_0+1)})^k\int_{B_t} \sum_{i=1}^mG\left(\dfrac{|u_i - \vartheta|}{t-s}\right) \ dx \\
        && \hspace{4cm}\lambda|B_t|\sum_{k=0}^{j-1}\gamma^k.
    \end{eqnarray*}
Finally, choosing $\tau=\tau(g_0) \in (0,1)$ in a such way that $\gamma\tau^{-(g_0+1)}<1$, and passing to the limit of $j\to \infty$ we get
    $$\int_{B_s} \sum_{i=1}^m G(|\nabla u_i|) \ dx \le \dfrac{(g_0+1)^2}{(1-\tau)^{g_0+1}(1-\gamma\tau^{-(g_0+1)})}\int_{B_t} \sum_{i=1}^mG\left(\dfrac{|u_i - \vartheta|}{t-s}\right) \ dx + \dfrac{1}{1-\gamma}\lambda|B_t|,$$
which proves the proposition.
\end{proof}

Next, we have the following.

\begin{corollary}\label{MainCorollary}
    Suppose $G\in \mathcal{G}(\delta,g_0)$, and let ${\bf u} = (u_1, \dots, u_m)$ be a $(\kappa,\beta)$-almost-minimizer of $\mathcal{J}_G$ in $\Omega$, with some positive constant $\kappa \le \kappa_0$ and exponent $0<\beta<1$. Furthermore, assume that $g_0=\delta+\varepsilon$, where $\varepsilon>0$ is given in Theorem \ref{gradHold}, and $B_1({\bf 0}) \Subset \{{\bf |u|}>0\}$. Then,
        $$
        |\nabla {\bf u}({\bf 0})|\le \mathrm{C}\cdot \Big( \|{\bf u}\|_{L^\infty(B_1({\bf0}),\mathbb{R}^m)} + \lambda\Big),
        $$
    for a universal constant $\mathrm{C}=\mathrm{C}(n,m,\kappa_0,\beta,\delta,g_0)>0$.
\end{corollary}
\begin{proof}
    By applying Theorem \ref{gradHold}, we obtain the following local  H\"{o}lder estimate for the gradient:
        $$\|{\bf u}\|_{C^{1,\gamma}(B_{1/4}({\bf 0});\mathbb{R}^m)} \le \mathrm{C}\left( \sum_{i=1}^m\int_{B_{1/2}({\bf 0})} G\big(\vert\nabla u_i\vert\big) \ dx + \lambda\right).$$
    On the other hand, making use of the Caccioppoli-type inequality stated in Proposition \ref{CaccioIne} with $\vartheta = 0$, we derive the estimate
        \begin{eqnarray*}
        \|{\bf u}\|_{C^{1,\gamma}(B_{1/4}({\bf 0});\mathbb{R}^m)} &\le& \mathrm{C}\left( \sum_{i=1}^m\int_{B_{1}({\bf 0})} G\big(\vert u_i\vert\big) \ dx + \lambda\right) \\
        &\le& \mathrm{C}\left( |B_1({\bf 0})|\sum_{i=1}^m G\big(\Vert u_i\Vert_{L^\infty(B_1({\bf 0}),\mathbb{R}^m)}\big) + \lambda\right) \\
        &\le& \mathrm{C} \Big(G\big(\Vert {\bf u}\Vert_{L^\infty(B_1({\bf 0}),\mathbb{R}^m)}\big)  + \lambda\Big),
        \end{eqnarray*}
    where in the last inequality we used the monotonicity of $G$, $(G_2)$ and the subadditivity of the norm. This concludes the proof of the corollary.
\end{proof}

\subsection{Linear growth at free boundary points}

The following theorem establishes the compactness of the non-degenerate subclass $\mathcal{G}_{\nu}(\delta, g_0)$ (we refer the reader to Definition \ref{defclasses}) with respect to the appropriate topologies.

\begin{theorem}[{\bf Compactness in $\mathcal{G}_\nu(\delta, g_0, \varepsilon_0)$ - \cite[Theorem 6.1]{BM}}]\label{Thm-Compactness-G_j}
Let $0 < \nu \leq 1$ and $\tau, \mathfrak{L}_0 > 0$. Define
\[
\mathscr{F}_\tau(\rho_0, \mathfrak{L}_0) := \left\{ 
G \in \mathcal{G}_{\nu}(\delta, g_0, \rho_0) \;:\; G(\tau) \leq \mathfrak{L}_0
\right\}.
\]
If $\{G_j\}_{j \in \mathbb{N}} \subset \mathscr{F}_\tau(\rho_0, L)$, then there exists a subsequence, still denoted by $\{G_j\}_{j \in \mathbb{N}}$, and a function $G \in \mathcal{G}(\delta, g_0) \cap C^{2,\beta_0}((0, \infty))$ such that
\[
G_j \to G \text{ in the } C^{2,\gamma} \text{topology on compact subsets of } (0, \infty) \text{ for every } \gamma \in (0, \beta_0),
\]
and
\[
G_j \to G \text{ in the } C^{1} \text{ topology on compact subsets of } [0, \infty).
\]
\end{theorem}

\medskip

Many estimates for almost-minimizers of functionals of the type $\mathcal{J}_G$ involve universal constants, namely constants that depend only on the ellipticity parameters $\delta$ and $g_0$. In some cases, these constants also rely on the value of $G(1)$. 

For this reason, our almost-minimizers must admit rescalings that normalize the value $G(1)$. This property is essential to enable a compactness argument based on the previous theorem, which will be employed in the blow-up analysis presented in the next proposition.

For $\varrho > 0$, we define the following normalized rescalings:
\[
G_{\varrho}(t) := \frac{G(\varrho t)}{\varrho g(\varrho)} 
\quad \text{and} \quad 
G^*_{\varrho}(t) := G(\varrho t).
\]
The next proposition collects key properties of these rescalings.

\begin{proposition}[{\bf Scaling properties - \cite[Proposition 6.1]{BM}}]\label{prop:scaling} The following properties hold:
\begin{itemize}
    \item[(S-1)] If $G \in \mathcal{G}(\delta, g_0)$, then $G_{\varrho} \in \mathcal{G}\left(\delta, g_0, (1 + g_0)^{-1}\right)$ and $G_{\varrho}(1) \leq 1$;
    \item[(S-2)] If $G \in \mathcal{G}_{\nu}(\delta, g_0)$, then $G_{\varrho} \in \mathcal{G}_{\nu}\left(\delta, g_0, (1 + g_0)^{-1}\right)$ and $G_{\varrho}(1) \leq 1$;
    \item[(S-3)] If $G \in \mathcal{G}(\delta, g_0, \rho_0)$ and $\varrho \geq 1$, then $G^*_{\varrho} \in \mathcal{G}(\delta, g_0, \rho_0)$;
    \item[(S-4)] If $G \in \mathcal{G}_{\nu}(\delta, g_0, \rho_0)$ and $\varrho \geq 1$, then $G^*_{\varrho} \in \mathcal{G}_{\nu}(\delta, g_0, \rho_0)$.
\end{itemize}
\end{proposition}

\bigskip
In the next proposition, we employ the following notation:
\[
u_{r, \mathrm{T}}(x) := \frac{u(r x)}{\mathrm{T}}.
\]

\begin{proposition}\label{Prop4.2}
Let $\{G_j\}_{j \in \mathbb{N}}$ be a sequence such that $G_j  \in \mathcal{G}_{\nu}(\delta, g_0, \rho_0)$, for any $j\in\mathbb{N}$. Consider $\mathbf{u}^j = (u^j_1, \dots, u^j_m)$ a sequence of bounded  $(\kappa,\beta)$-almost-minimizers of $\mathcal{J}_{G_j}$ in $B_2$. Define
\[
\mathbf{v}^j(x) := \mathbf{u}^j_{r_j, \mathrm{T}_j}(x) = \frac{\mathbf{u}^j(r_j x)}{\mathrm{T}_j}, \quad \text{in } B_{2R},
\]
where $0 < R < \frac{1}{r_j}$, $r_j \to 0$ as $j \to \infty$, and $\mathrm{T}_j > 0$. Then, $\mathbf{v}^j = (v^j_1, \dots, v^j_m)$ is an almost-minimizer (concerning its own boundary values) of the scaled functional
\[
\widehat{\mathcal{J}}_{\widehat{G}_j}(\mathbf{v}^j; U) := \int_U \left(\sum_{i=1}^m \widehat{G}_j(|\nabla v^j_i|) +  \widehat{\lambda}_j \chi_{\{|\mathbf{v}^j| > 0\}}\right) \, dx,
\]
with constant $\widehat{\kappa} = \kappa r_j^\beta$ and exponent $\widehat{\beta} = \beta$, where 
$$
\widehat{G}_j(t) \coloneqq G_{\sigma_j}(t) = \frac{G_j(\sigma^{-1}_j t)}{\mathfrak{a}_j}, \,\,\,\sigma_j := \frac{r_j}{\mathrm{T}_j}, \quad \widehat{\lambda}_j  \coloneqq \frac{\lambda_j}{\mathfrak{a}_j},\quad \mathfrak{a}_j:= \sigma^{-1}_j g_j(\sigma^{-1}_j) \quad  \text{and} \quad G^{\prime}_j(t) = g_j(t).
$$

Additionally, if $\|\mathbf{v}^j\|_{L^\infty(B_{2R})} \leq \mathrm{M}$
for any fixed $0 < R < \frac{1}{r_j}$ and for some constant $\mathrm{M} = \mathrm{M}(R) > 0$, then, up to a subsequence, the following holds:
\begin{enumerate}[(i)]
    \item $v^j_i \rightharpoonup v^\infty_i$ weakly in $W^{1,G}(B_R)$, and also strongly in $C^{0, \alpha}(B_R)$ for any $\alpha < \frac{\delta}{g_0}$.    
    \item Furthermore, if $\sigma_j = \frac{r_j}{\mathrm{T}_j} \to 0$ and $G_j \to G$ (in an appropriate topology) as $j \to \infty$, then  $\nabla v_i^j (x) \to \nabla v_i^\infty(x)$ a.e. in $B_R$, as $j\to \infty$, and $v^\infty_i$ is a $g$-harmonic function in $B_R$ with $G^{\prime}(t) = g(t)$.
\end{enumerate}
\end{proposition}

\begin{proof} We split the proof into three steps:
\begin{enumerate}
    \item[\checkmark] $\mathbf{v}^j$ is an almost-minimizer of the functional $\widehat{\mathcal{J}}_{\widehat{G}_j}$

    {\bf Claim:} $\mathbf{v}^j$ is an almost-minimizer of the functional $\widehat{\mathcal{J}}_{\widehat{G}_j}$, with constant $\widehat{\kappa}_j = \kappa r_j^\beta$ and exponent $\widehat{\beta} = \beta$, i.e.,
\begin{equation}\label{eq:almost_min}
\widehat{\mathcal{J}}_{\widehat{G}_j}\left(\mathbf{v}^j; B_\rho(x_0)\right) \leq \left(1 + \widehat{\kappa}_j \rho^{\widehat{\beta}}\right) \widehat{\mathcal{J}}_{\widehat{G}_j}\left(\mathbf{w}; B_\rho(x_0)\right),
\end{equation}
for every ball $B_\rho(x_0) \subset B_{\frac{1}{r_j}}$ and for every $\mathbf{w} \in W^{1,G}(B_\rho(x_0); \mathbb{R}^m)$ satisfying $\mathbf{w} = \mathbf{v}^j$ on $\partial B_\rho(x_0)$.

Indeed, by the almost-minimality property of $\mathbf{u}^j$, we have
\[
\mathcal{J}_{G_j}\left(\mathbf{u}^j; B_{r_j \rho}(y_0)\right) \leq \left(1 + \kappa (r_j \rho)^\beta\right) \mathcal{J}_{G_j}\left(\mathbf{w}^j; B_{r_j \rho}(y_0)\right),
\]
where $y_0 = r_j x_0$ and $\mathbf{w}^j(x) := \mathrm{T}_j \mathbf{w}\left(\frac{x}{r_j}\right)$. 

Now, by expanding the definition of $\mathcal{J}_{G_j}$, this inequality reads
\begin{equation}\label{eq:almost_min_original}
\int \limits_{B_{r_j \rho}(y_0)} \left(\sum_{i=1}^m G_j(|\nabla u^j_i(y)|) + \lambda_j \chi_{\{|\mathbf{u}^j| > 0\}}\right)\,dy \leq \left(1 + \kappa r_j^\beta \rho^\beta\right)\int \limits_{B_{r_j \rho}(y_0)} \left(\sum_{i=1}^m G_j(|\nabla w^j_i(y)|) + \lambda_j \chi_{\{|\mathbf{w}^j| > 0\}}\right)\,dy.
\end{equation}

Next, by applying the change of variables $y = r_j x$, we obtain
$$
\begin{array}{rcl}
   \displaystyle \widehat{\mathcal{J}}_{\widehat{G}_j}\left(\mathbf{v}^j; B_\rho(x_0)\right) &  =  & \displaystyle\int_{B_\rho(x_0)} \left(\sum_{i=1}^m \widehat{G}_j(|\nabla v^j_i|) +  \widehat{\lambda}_j \chi_{\{|\mathbf{v}^j| > 0\}}\right) \, dx\\
     & = &  \displaystyle \frac{1}{\mathfrak{a}_j} \int_{B_\rho(x_0)} \left(\sum_{i=1}^m G_j(\sigma_j^{-1}|\nabla v^j_i|) +  \lambda_j \chi_{\{|\mathbf{v}^j| > 0\}}\right) \, dx\\
     & = &  \displaystyle \frac{1}{\mathfrak{a}_j} \int_{B_\rho(x_0)} \left(\sum_{i=1}^m G_j(|\nabla u^j_i(r_jx)|) +  \lambda_j \chi_{\{|\frac{\mathbf{u}^j}{\mathrm{T}_j}| > 0\}}(r_j x)\right) \, dx\\
     & = & \displaystyle \frac{r_j^{-n}}{\mathfrak{a}_j} \int_{B_{r_j \rho}(y_0)} \left(\sum_{i=1}^m G_j(|\nabla u^j_i(y)|) +  \lambda_j \chi_{\{|\mathbf{u}^j| > 0\}}(y)\right) \, dy.
\end{array}
$$
Similarly, we have
$$
\widehat{\mathcal{J}}_{\widehat{G}_j}\left(\mathbf{w}; B_\rho(x_0)\right) = \displaystyle  \frac{r_j^{-n}}{\mathfrak{a}_j}\int_{B_{r_j \rho}(y_0)} \sum_{i=1}^m G_j(|\nabla w^j_i(y)|) + \lambda_j \chi_{\{|\mathbf{w}^j| > 0\}}(y)\,dy.
$$

%in the right-hand side of \eqref{eq:almost_min_original}, we obtain
%$$
%\begin{array}{rcl}
%   \displaystyle \int_{B_{r_j \rho}(y_0)} \sum_{i=1}^m G_j(|\nabla w^j_i(y)|) + \lambda_j \chi_{\{|\mathbf{w}^j| > 0\}}(y)\,dy & = & \displaystyle  r_j^n \int_{B_\rho(x_0)} \sum_{i=1}^m G_j(| \nabla w^j_i(r_j x)|) + \lambda_j \chi_{\{|\mathbf{w}^j|> 0\}}(r_j x)\,dx %\\
%     &  = & \displaystyle  r_j^n \mathfrak{a}_j\int_{B_\rho(x_0)}  \sum_{i=1}^m G_j(\sigma_j^{-1}|\nabla w_i(x)|)dx\\
%     & + &   \displaystyle r_j^n \mathfrak{a}_j\int_{B_\rho(x_0)} \lambda_j \chi_{\{|\mathbf{w}|>0\}}(x)\,dx,
%\end{array}
%$$
%\begin{align*}
%&\int_{B_{r_j \rho}(y_0)} \sum_{i=1}^m \widehat{G}_j(|\nabla w_j^i(y)|) + \lambda_j \chi_{\{|w_j| > 0\}}(y)\,dy \\
%&= r_j^n \int_{B_\rho(x_0)} \sum_{i=1}^m \widehat{G}_j(\left| \nabla w_j^i(r_j x) \right|) + \lambda_j \chi_{\{|w_j|> 0\}}(r_j x)\,dx \\
%&= r_j^n \int_{B_\rho(x_0)} \sigma^{-1}_j g_j(\sigma^{-1}_j) \sum_{i=1}^m \widehat{G}_j(|\nabla w^i(x)|) + \lambda_j \chi_{\{|w|>0\}}(x)\,dx,
%\end{align*}
%and a similar formula holds for $\mathbf{u}^j$ and $\mathbf{v}^j = \mathbf{u}^j_{r_j, \mathrm{T}_j}$ replacing $w_j$ and $w$.

Therefore, by substituting such identities into \eqref{eq:almost_min_original}, we deduce
$$
\displaystyle \int_{B_\rho(x_0)} \sum_{i=1}^m \widehat{G}_j(|\nabla v_j^i|) + \widehat{\lambda}_j \chi_{\{|v_j| > 0\}}\,dx \leq \left(1 + \kappa r_j^\beta \rho^\beta\right)\int_{B_\rho(x_0)} \sum_{i=1}^m \widehat{G}_j(|\nabla w^i|) + \widehat{\lambda}_j \chi_{\{|w| > 0\}}\,dx,
$$
which is precisely the desired result \eqref{eq:almost_min}.
    \item[\checkmark] Proof of item (i)

    Using the assumption $\|\mathbf{v}^j\|_{L^\infty(B_{2R})} \leq \mathrm{M}$ together with Proposition \ref{CaccioIne}, we obtain uniform $W^{1,G}(B_R)$ estimates for $v^j_i$, for $j$ sufficiently large. Consequently, passing to a subsequence if necessary, we obtain
\[
v^j_i \rightharpoonup v^\infty_i \quad \text{weakly in } \quad W^{1,G}(B_R).
\]
Moreover, by applying the uniform H\"{o}lder estimates from Theorem \ref{Holdereg}, it follows that
\[
v^j_i \to v^\infty_i \quad \text{in } \quad  C^{0,\alpha}(B_R),
\]
for any $\alpha \in  \left(0, \frac{\delta}{g_0}\right)$, thereby concluding the proof of part (i).

 \item[\checkmark] Proof of item (ii)

Define $z^j_i$ to be the $\widehat{g}_j$-harmonic replacement of $v^j_i$ in $B_R$, where $\widehat{g}_j=\widehat{G}^{\prime}_j$. Then, by Lemma \ref{controlVg}, we have
\begin{equation}\label{eq:harmonic_approx}
\int_{B_R} \left| \mathbf{V}_{\widehat{G}_j}\left(\nabla v^j_i\right) - \mathbf{V}_{\widehat{G}_j}\left(\nabla z^j_i\right) \right|^2 dx \leq \mathrm{c} \kappa r_j^\beta R^\beta \int_{B_R} \sum_{i=1}^m \widehat{G}_j(|\nabla v^j_i|) dx + \mathrm{c} \widehat{\lambda}_j R^n.
\end{equation}
%where $\sigma_j = \frac{r_j}{\mathrm{T}_j}$.

Further, by $(G_3)$ applied to $\widehat{G}_j$ and \cite[see Theorem 2.3]{MW08} we have
    \begin{eqnarray*}
        \int_{B_R} \left| \mathbf{V}_{\widehat{G}_j}\left(\nabla v^j_i\right) - \mathbf{V}_{\widehat{G}_j}\left(\nabla z^j_i\right) \right|^2 dx &\ge& \mathrm{C}\int_{B_R} \Big(\widehat{G}^{\prime}_j\big(|\nabla v_i^j|\big)|\nabla v_i^j| - \widehat{G}^{\prime}_j\big(|\nabla z_i^j|\big)|\nabla z_i^j| \Big)\, dx \\
        &\ge& \mathrm{C}\int_{B_R} \Big(\widehat{G}_j\big(|\nabla v_i^j|\big) - \widehat{G}_j\big(|\nabla z_i^j|\big) \Big)\, dx \\ 
        &\ge& \mathrm{C} \int_{B_R} \widehat{G}_j\big(|\nabla v_i^j - \nabla z_i^j|\big)\, dx,
    \end{eqnarray*}
where $C>0$ depends only on $n,\delta$, and $g_0$. Thus, passing to the limit $j \to \infty$, by \eqref{eq:harmonic_approx} we deduce that, up to a subsequence,
    \begin{equation}\label{convmodstrong}
        \int_{B_R} \widehat{G}_j\big(|\nabla v_i^j-\nabla z_i^j|\big)\, dx \to 0,
    \end{equation}
as $j \to \infty$, for each $i=1,\dots,m$. Since $\widehat{g}_j$ satisfies the growth condition given by inequality \eqref{Ga}, with the same constants $\delta$ and $g_0$, it is well known that the embedding $W^{1,\widehat{G}_j} \hookrightarrow W^{1,\delta+1}$ is continuous. As a consequence, and using the fact that $v_i^j - z_i^j = 0$ on $\partial B_R$ (in the trace sense), we conclude that
\begin{equation}\label{strongc}
v^j_i - z^j_i \to 0 \quad \text{in } W^{1,\delta+1}(B_R).
\end{equation}
This strong convergence implies that
    \begin{equation}\label{aeconvgrad}
        \nabla \big(v_i^j - z_i^j\big) (x) \to 0\quad \mbox{a.e. in} \ B_R,
    \end{equation}
passing to a subsequence, if necessary. Next, since $z_i^j$ is $\widehat{g}_j$-harmonic we may apply the classical Lierberman's result (see \cite[Theorem 1.7]{L}) to ensure that there exists $\mathrm{C}=\mathrm{C}(n,\delta,g_0)>0$ such that $\|z_i^j\|_{C_{loc}^{1,\alpha_0}(B_R)}\le \mathrm{C}$, for some $\alpha_0\in(0,1)$. Then, we can find $z_i^\infty \in C_{loc}^{1,\alpha_0/2}(B_R)$ such that, up to a subsequence, $\nabla z_i^j \to \nabla z_i^\infty$ uniformly on compact sets of $B_R$. By the strong convergence \eqref{strongc} in $W^{1,\delta+1}(B_R)$, we get that $z_i^\infty \equiv v_i^\infty$. Therefore, \eqref{aeconvgrad} yields
    $$\nabla v_i^j(x) \to \nabla v_i^\infty(x), \quad \mbox{a.e.}\ x\in B_R,$$
as $j\to \infty$, proving the first part of item (ii). 

For the second part, we must prove that $v_i^\infty$ is $g$-harmonic. To do so, we use the fact that ${\bf v}^j$ is a $(\kappa,\beta)$-almost-minimizer of $\widehat{\mathcal{J}}_{\widehat{G}_j}$ to get

\begin{equation}\label{almostdef}
    \int_{B_1} \left(\sum_{i=1}^m \widehat{G}_j\big(|\nabla v_i^j|\big) + \hat{\lambda}_j\chi_{\{|{\bf v}^j|>0\}}\right) \, dx \le \big(1+\kappa r_j^\beta\big)  \int_{B_1} \left(\sum_{i=1}^m \widehat{G}_j\big(|\nabla (v_i^j+\varphi_i)|\big) + \hat{\lambda}_j\chi_{\{|{\bf v}^j+{\bf \varphi}|>0\}}\right) \, dx,
\end{equation}
for any $\varphi \in C_0^\infty(B_1,\mathbb{R}^m)$. Next, by the $C^{1,\frac{\alpha_0}{2}}$-estimate of $z_i^j$ we know that $\nabla z_i^j$ is uniformly bounded in $B_1$. Then, by uniform convergence
    \begin{equation}\label{convstrongnabla}
        \int_{B_1} \big| \widehat{G}_j\big(|\nabla z_i^j|\big) - \widehat{G}_j\big(|\nabla v_i^\infty|\big)\big| \, dx \to 0,
    \end{equation}
as $j \to \infty$, for any $i=1,\dots,m$. Further, by $(G_2)$ we obtain
    $$\widehat{G}_j\big(|\nabla v_i^j|\big) \le 2^{g_0}(1+g_0)\Big(\widehat{G}_j\big(|\nabla \big(v_i^j-z_i^j\big)|\big) + \widehat{G}_j\big(|\nabla z_i^j|\big)\Big),$$
which implies by \eqref{convmodstrong} and \eqref{convstrongnabla} that there exists $h \in L^1(B_1,\mathbb{R}^m)$ such that
    \begin{equation}\label{dominated}
        \widehat{G}_j\big(|\nabla v_i^j|\big) \le h \quad \mbox{a.e. in} \ B_1,
    \end{equation}
for each $i=1,\dots,m$. Thus, Lebesgue's Dominated Convergence Theorem ensures that
    $$\int_{B_1} \sum_{i=1}^m \widehat{G}_j\big(|\nabla v_i^j|\big) \, dx \to \int_{B_1} \sum_{i=1}^m G\big(|\nabla v_i^\infty|\big) \, dx \quad \mbox{as} \ j \to \infty.$$
Furthermore, by \eqref{dominated} we get
    $$\widehat{G}_j\big(|\nabla \big(v_i^j+\varphi_i\big)(x)|\big) \le h_\varphi(x),$$
a.e. in $B_1$, where $\displaystyle h_\varphi \coloneqq \mathrm{C}(g_0)\big[h + (1+\sup_{B_1} |\nabla \varphi|)^{1+g_0}\big] \in L^1(B_1,\mathbb{R}^m)$. Thus, we may apply Lebesgue's Dominated Convergence Theorem again to get
    $$\int_{B_1} \sum_{i=1}^m \widehat{G}_j\big(|\nabla \big(v_i^j+\varphi_i\big)|\big) \, dx \to \int_{B_1} \sum_{i=1}^m G\big(|\nabla \big(v_i^\infty+\varphi_i\big)|\big) \, dx \quad \mbox{as} \ j \to \infty.$$
In addition, we can see easily that
    $$\int_{B_1} \hat{\lambda}_j \chi_{\{|{\bf v}^j+{\bf \varphi}|>0\}} \, dx \to 0 \quad \mbox{and} \quad \int_{B_1} \hat{\lambda}_j \chi_{\{|{\bf v}^j|>0\}}\to 0,$$
both as $j \to \infty$. Therefore, passing to the limit of $j\to\infty$ in \eqref{almostdef}, we have
    $$\int_{B_1} \sum_{i=1}^m G\big(|\nabla v_i^\infty|\big) \, dx \le \int_{B_1} \sum_{i=1}^m G\big(|\nabla \big(v_i^\infty+\varphi_i\big)|\big) \, dx,$$
for any $\varphi \in C_0^\infty(B_1,\mathbb{R}^m)$, where  $g(t) = G^{\prime}(t)$ and $\displaystyle G(t) = \lim_{j \to \infty}\widehat{G}_j(t)$ (according Theorem \ref{Thm-Compactness-G_j}), which proves the $g$-harmonicity of ${\bf v}^\infty$, as desired.

 This completes the proof of part (ii) and therefore concludes the proof of the proposition.
 
\end{enumerate}
\end{proof}

\medskip
Before obtaining the Lipschitz continuity of almost-minimizers of $\mathcal{J}_{G}$, we are going to present the following key result concerning the linear growth of almost-minimizers at free boundary points.

\begin{proposition}\label{MainProp02}
Assume that $G \in \mathcal{G}_{\nu}(\delta, g_0, \rho_0)$. Consider ${\bf u}=(u_1,\dots, u_m)$ a $(\kappa,\beta)$-almost-minimizer of $\mathcal{J}_{G}$ in $B_1(x_0)$, where $x_0 \in \mathfrak{F}(\mathbf{u})$. Furthermore, suppose that
\[
\sup_{B_1(x_0)} |\mathbf{u}| \leq \mathrm{M}_0.
\]
Then, there exists a universal constant $\mathrm{C} \geq 1$, depending only on $\mathrm{M}_0$, such that
\begin{equation}\label{Lips decay}
|\mathbf{u}(x)| \leq \mathrm{C} \mathrm{M}_0 |x - x_0|
\end{equation}
for all $x \in B_r(x_0)$ and any $0 < r < 1$.
\end{proposition}

\begin{proof}
First, observe that, without loss of generality, we may assume that $x_0 = \mathbf{0}$.  Moreover, by combining discrete iterative techniques with continuous reasoning (see, for instance, \cite{CKS}), it is well established that proving estimate \eqref{Lips decay} is equivalent to verifying the existence of a universal constant $\tilde{\mathrm{C}}>1$ such that, for all $j\in \mathbb{N}$, the following holds
\begin{equation} \label{eq:22}
\mathcal{S}(k + 1, \mathbf{u}) \leq \max \left\{ \frac{\widetilde{\mathrm{C}} \mathrm{M}_0}{2^{k+1}}, \frac{\mathcal{S}(k, \mathbf{u})}{2} \right\},
\end{equation}
where $\displaystyle \mathcal{S}(k, \mathbf{u}) := \sup_{B_{2^{-k}}} |\mathbf{u}|$. Indeed, one can deduce inductively from \eqref{eq:22} that
\[
\mathcal{S}(k, u) \leq \widetilde{\mathrm{C}} \mathrm{M}_0 2^{-k}.
\]
Now, for any $r \in (0, 1]$, choose $k \geq 0$ such that $2^{-(k+1)} < r \leq 2^{-k}$. Hence,
\[
\| \mathbf{u} \|_{L^\infty(B_r)} \leq \| \mathbf{u} \|_{L^\infty(B_{2^{-k}})} = \mathcal{S}(k, \mathbf{u}) \leq \widetilde{\mathrm{C}} \mathrm{M}_0 2^{-k} \leq 2\widetilde{\mathrm{C}} \mathrm{M}_0 2^{-(k+1)} \leq 2\widetilde{\mathrm{C}} \mathrm{M}_0 r,
\]
which establishes the desired estimate.

It remains to prove \eqref{eq:22}. Suppose, by contradiction, that it fails. Then, for each $j \in \mathbb{N}$, there exist $\mathbf{u}^j$ and $G_j \in \mathcal{G}_{\nu}(\delta, g_0,\rho_0)$ and $0 \leq \lambda_j \leq \Lambda< \infty$ in $B_1$ such that %uj is a minimizer o
$\mathbf{u}^j$ is an almost-minimizer of $\mathcal{J}_{G_j}$ in $B_1$, i.e.,
$$
\mathcal{J}_{G_j} (\mathbf{u}^j; U) := \int_U \left(\sum_{i=1}^m G_j(|\nabla u^j_i|)+ \lambda_j \chi_{\{|\mathbf{u}^j| > 0\}}\right) \, dx, \quad \text{with} \quad G^{\prime}_j(t) = g_j(t),
$$
and there exist integers $k_j \to \infty$ as $j \to \infty$, such that
\begin{equation} \label{eq:23}
\mathcal{S}(k_j + 1, \mathbf{u}^j) > \max \left\{ \frac{j\mathrm{M}_0}{2^{k_j + 1}}, \frac{\mathcal{S}(k_j, \mathbf{u}^j)}{2} \right\}.
\end{equation}
Here, $\displaystyle \mathcal{S}(k_j, \mathbf{u}^j) := \sup_{B_{2^{-k_j}}} |\mathbf{u}^j|$ with $\| u_j \|_{L^\infty(B_1)} \leq \mathrm{M}_0$. Moreover, observe that the condition $\| \mathbf{u}^j \|_{L^\infty(B_1)} \leq \mathrm{M}_0$ indeed forces $k_j \geq \frac{\ln(j)}{\ln(2)} -1 \to \infty$ as $j \to \infty$.

Now, we define the rescaled auxiliary function
\[
\mathbf{v}^j(x) := \frac{\mathbf{u}^j(2^{-k_j} x)}{\mathcal{S}(k_j + 1, \mathbf{u}^j)} \quad \text{in } \quad B_{1},
\]
and set
\[
\sigma^{\ast}_j := \frac{2^{-k_j}}{\mathcal{S}(k_j + 1, \mathbf{u}^j)}.
\]
Thus, by \eqref{eq:23}, it follows that $\sigma^{\ast}_j \leq \frac{2}{j\mathrm{M}_0} \to 0$ as $j \to \infty$. Furthermore, once again by \eqref{eq:23} we have
\begin{itemize}
    \item[\checkmark] $0 \leq v^j_i \leq 2$ in $B_{1}$;
    \item[\checkmark] $v^j_i(\mathbf{0}) = 0$;
    \item[\checkmark] $\displaystyle \sup_{B_{1/2}} v^j_i = 1$.
\end{itemize}

Now, consider the rescaled energy functional
\[
\widehat{\mathcal{J}}_{\widehat{G}_j} (\mathbf{v}^j; U) := \int_U \left(\sum_{i=1}^m \widehat{G}_j(|\nabla v^j_i|) +  \widehat{\lambda}_j \chi_{\{|\mathbf{v}^j| > 0\}}\right) \, dx,
\]
where 
$$
\widehat{G}_j(t) \coloneqq G_{\sigma^{\ast}_j}(t) = \frac{G_j((\sigma^{\ast}_j)^{-1} t)}{(\sigma^{\ast}_j)^{-1} g_j((\sigma^{\ast}_j)^{-1})} \quad \text{and} \quad  \widehat{\lambda}_j  \coloneqq \frac{\lambda_j}{(\sigma^{\ast}_j)^{-1} g_j((\sigma^{\ast}_j)^{-1})}.
$$

Now, note that from Proposition \ref{prop:scaling} (Statement (S-2)) we have $\widehat{G}_j \in \mathcal{G}_{\nu}\big(\delta, g_0, (1+g_0)^{-1}\big)$. Moreover,  from Definition \ref{Def-N-function} (Statement (iii)), and Statement \ref{Statement} --Property $(g_3)$--, thus we conclude that $\displaystyle \lim_{j \to \infty} \widehat{\lambda}_j = 0$.

Moreover, since $\| \mathbf{v}^j \|_{L^\infty(B_1)} \leq 2$ by \eqref{eq:23}, it follows from Proposition \ref{Prop4.2} that $\mathbf{v}^j = (v^j_1, \dots, v^j_m)$ is an almost-minimizer of $\widehat{\mathcal{J}}_{\widehat{G}_j}$ with constant $\widehat{\kappa} = \kappa (2^{-k_j})^\beta$ and exponent $\widehat{\beta} = \beta$. 

Hence, passing to the limit as $j \to \infty$, the limiting function $\displaystyle v^\infty_i(x) := \lim_{j \to \infty} v^j_i(x)$ (for each $1\leq i \leq m$) satisfies in the weak sense (according Proposition \ref{Prop4.2} -  (i)-(ii), and  Theorem \ref{Thm-Compactness-G_j})
\[
\Delta_g v^\infty_i = \mbox{div} \left(g(|\nabla v^\infty_i|)\dfrac{\nabla v^\infty_i}{|\nabla v^\infty_i|}\right)=  0 \quad \text{in } B_{3/4} \quad \text{with} \quad g(t) = G^{\prime}(t), \,\,\,\text{and}\quad  \displaystyle G(t) = \lim_{j \to \infty} \widehat{G}_j(t).
\]
On the other hand, the following properties hold:
\begin{itemize}
    \item[\checkmark] $0 \leq v^\infty_i \leq 2$ in $B_{3/4}$;
    \item[\checkmark] $v^\infty_i(\mathbf{0}) = 0$;
    \item[\checkmark] $\displaystyle \sup_{B_{1/2}} v^\infty_i = 1$.
\end{itemize}
Finally, such properties contradict the homogeneous Harnack inequality for $g$-harmonic profiles (see e.g. Theorem \ref{Harnack-Ineq}), thereby completing the proof.
\end{proof}

\subsection{Lipschitz regularity: Proof of Theorem \ref{MainThm03}}

Finally, we are in a position to address the proof of Theorem \ref{MainThm03}.

\begin{proof}[{\bf Proof of Theorem \ref{MainThm03}}] Let $\mathbf{u}$ be an almost minimizer of $\mathcal{J}_G$ in $U$, with some constant $\kappa \leq \kappa_0$ and exponent $\beta$. Consider an open subset $\widetilde{U} \Subset U$, and let
\[
r_0 := \frac{1}{4} \min\left\{2, \operatorname{dist}(\widetilde{U}, \partial U)\right\}.
\]
Define the interior set
\[
U_{r_0} := \left\{x \in U : \operatorname{dist}(x, \partial U) \geq r_0 \right\}.
\]
By Theorem \ref{Holdereg}, we know that $\mathbf{u} \in C^{0, \eta}(U_{r_0}; \mathbb{R}^m)$ for some $\eta \in \left(0, \frac{\delta}{g_0}\right)$. Let
\[
\mathrm{M} := \|\mathbf{u}\|_{L^\infty(U_{r_0}; \mathbb{R}^m)}.
\]

For an arbitrary point $x_0 \in \widetilde{U} \cap \{|\mathbf{u}| > 0\}$, in order to estimate $|\nabla \mathbf{u}(x_0)|$, we distinguish two scenarios:

\medskip

\noindent\textbf{Case I:} If $\mathbf{d} := \operatorname{dist}(x_0, \mathfrak{F}(\mathbf{u})) \leq r_0$:

\smallskip

Choose a point $y_0 \in \partial B_{\mathbf{d}}(x_0) \cap \mathfrak{F}(\mathbf{u})$. Then, according to Proposition \ref{MainProp02}, for any $x \in B_d(x_0)$, it holds that
\[
|\mathbf{u}(x)| \leq \mathrm{C} \mathrm{M} |x - y_0| \leq 2\mathrm{C}\mathrm{M} \mathbf{d}.
\]
Note that $B_{2\mathbf{d}}(y_0) \subset U_{r_0}$, and therefore $|\mathbf{u}| \leq \mathrm{M}$ in $B_{2\mathbf{d}}(y_0)$. Consider the scaled function
\[
\mathbf{u}_{\mathbf{d}}(x) := \frac{\mathbf{u}(x_0 + \mathbf{d} x)}{\mathbf{d}},
\]
which is an almost minimizer of the functional
\[
\mathbf{v} \mapsto \int_{B_1} \left(\sum_{i=1}^m G|\nabla v_i|) + \lambda\chi_{\{|{\bf v}|>0\}}(x)\right) \, dx,
\]
with constant $\kappa_{\mathbf{d}}:=\kappa d^\beta$ and exponent $\beta$, in $B_1$. Furthermore, $|\mathbf{u}_{\mathbf{d}}| \leq 2\mathrm{C} \mathrm{M}$. By Corollary \ref{MainCorollary}, it follows that
\[
|\nabla \mathbf{u}(x_0)| = |\nabla \mathbf{u}_{\mathbf{d}}(\mathbf{0})| \leq \widetilde{\mathrm{C}},
\]
where $\widetilde{\mathrm{C}}>0$ depends only on $\delta, g_0, m, n, \kappa_0 r_0^\beta, \beta, G^{-1}(\lambda)$, and $\mathrm{C\cdot M}$.

\medskip

\noindent\textbf{Case II:} If $\mathbf{d} \geq r_0$:

\smallskip

Define the scaled function
\[
\mathbf{u}_{r_0}(x) := \frac{\mathbf{u}(x_0 + r_0 x)}{r_0}.
\]
Then, $\mathbf{u}_{r_0}$ is an almost minimizer of the same functional with constant $\kappa_{r_0}:=\kappa r_0^\beta$ and exponent $\beta$, in $B_1$. Moreover, it satisfies
\[
\|\mathbf{u}_{r_0}\|_{L^\infty(B_1)} \leq \frac{\mathrm{M}}{r_0}.
\]
Therefore,
\[
|\nabla \mathbf{u}(x_0)| = |\nabla \mathbf{u}_{r_0}(\mathbf{0})| \leq \overline{\mathrm{C}}_0,
\]
where $\overline{\mathrm{C}}_0>0$ depends only on $\delta, g_0, m, n, \kappa_0 r_0^\beta, \beta, G^{-1}(\lambda)$, and $\frac{\mathrm{M}}{r_0}$. This finishes the proof of the Theorem.

\end{proof}

\subsection*{Acknowledgments}

\hspace{0.4cm}  Pedro Fellype Pontes was partially supported by NSFC (W2433017), and BSH (2024-002378). J.V. da Silva has received partial support from CNPq-Brazil under Grant No. 307131/2022-0, FAEPEX-UNICAMP (Project No. 2441/23, Special Calls - PIND - Individual Projects, 03/2023), and Chamada CNPq/MCTI No. 10/2023 - Faixa B - Consolidated Research Groups under Grant No. 420014/2023-3. Minbo Yang was partially supported by  National Natural Science Foundation of China (12471114), and Natural Science Foundation of Zhejiang Province (LZ26A010002).

\end{document}